\documentclass[11pt,A4]{amsart} 
\input colordvi
\usepackage{amssymb}
\usepackage{amsmath}
\usepackage{latexsym}

\newcommand{\wt}{\theta}
\newcommand{\lk}{\left}
\newcommand{\HH}{H} 
\newcommand{\mZ}{\mathbb{Z}} 
\newcommand{\TT}{T} 
\newcommand{\HC}{\mathcal{H}} 
\newcommand{\ZZ}{\CH}
\newcommand{\uu}{}

\newcommand{\lqq}{\lefteqn}
\newcommand{\rk}{\right}
\newcommand{\la}{{\langle}}
\newcommand{\tim}{{[0,\infty)}}
\newcommand{\rrint} {[0,\infty)}
\newcommand{\epsi} {{\epsilon,\sigma}}

\newcommand{\cadlag}{c{\'a}dl{\'a}g}

\newcommand{\ww}{{{R}}}

\newcommand{\uc}{u^{c}}

\newcommand{\neu}{{\gamma_1} }
\newcommand{\nb}{{\gamma_2} }

\newcommand{\barray}{\begin{array}{rcl}}
\newcommand{\earray}{\end{array}}

\newcommand{\Aand}{\mbox{ and }}

\newcommand{\pmat}{\begin{pmatrix}}
\newcommand{\epmat}{\end{pmatrix}}

\newcommand{\DEQS}{\begin{eqnarray*}}
\newcommand{\EEQS}{\end{eqnarray*}}
\newcommand{\DEQSZ}{\begin{eqnarray}}
\newcommand{\EEQSZ}{\end{eqnarray}}

\newcommand{\bcase}{\begin{cases}}
\newcommand{\ecase}{\end{cases}}

\newcommand{\ep}{{{ \epsilon }}}

\newcommand{\CO}{{{ \mathcal O }}}

\newcommand{\CX}{{{ \mathcal X }}}

\newcommand{\CK}{{{ \mathcal K }}}

\newcommand{\CD}{{{ \mathcal D }}}
\newcommand{\FA}{{{ \mathfrak{A}}}}
\newcommand{\CA}{{{ \mathcal A }}}

\newcommand{\ra}{{\rangle}}

\newcommand{\CI}{{{ \mathcal I }}}

\newcommand{\FB}{{{ \mathbb{F} }}}

\newcommand{\FF}{{{ \mathbb{F} }}}

\newcommand{\CZ}{{{ \mathcal Z }}}
\newcommand{\CC}{{{ \mathcal C }}}
\newcommand{\CH}{{{ \RR }}} 
\newcommand{\CCH}{{{ H }}}
\newcommand{\CS}{{{ \mathcal S }}}
\newcommand{\CG}{{{ \mathcal G }}}

\newcommand{\CB}{{{ \mathcal B }}}
\newcommand{\CR}{{{ \mathcal D }}}
\newcommand{\CM}{{{ \mathcal M }}}
\newcommand{\CP}{{{ \mathcal P }}}
\newcommand{\CF}{{{ \mathcal F }}}
\newcommand{\CL}{{{ \mathcal L }}}

\newcommand{\RR}{{\mathbb{R}}}

\newcommand{\QQ}{{\mathbb{Q}}}

\newcommand{\dd}{{\rm I \kern -0.2em D}}
\newcommand{\NN}{{\mathbb{N}}}

\newcommand{\PP}{{\mathbb{P}}}

\def\Re{\mathop{ \rm Re }}
\newcommand\sgn{\mathop{ \rm sgn }}

\newcommand{\EE}{ \mathbb{E} }
\newcommand{\vardelta}{ \rho}
\newcommand{\hh}{\rho}

\newcommand\del[1]{}
\newcommand\think[1]{}
\newcommand\new[1]{}
\newcommand\zus[1]{}

\newcommand\com[1]{\OliveGreen{#1}}
\newcommand\comd[1]{} 
\newcommand\Redd[1]{} 

\newcommand\coma[1]{\Blue{#1}}

\newcommand\red[1]{\Red{#1}}

\newtheorem{theorem}{Theorem}[section]

\newtheorem{definition}[theorem]{Definition}
\newtheorem{lemma}[theorem]{Lemma}

\newtheorem{coro}[theorem]{Corollary}
\newtheorem{example}[theorem]{Example}

\newtheorem{Notation}{Notation}

\newtheorem{hypo}{Hypothesis}
\newtheorem{remark}{Remark}
\begin{document}

\date{\today}

\title[\today ]{Controllability  and Qualitative properties of the solutions to SPDEs driven
by boundary L\'evy noise}

\author[EH \today]{
Erika Hausenblas,  Paul Andr\'e Razafimandimby }

\email[E. Hausenblas]{erika.hausenblas@unileoben.ac.at}
 \email[P. A. Razafimandimby]{paulrazafi@gmail.com}
\address[E. Hausenblas and P. A. Razafimandimby]{Department of Mathematics and Information Technology, Montanuniversity Leoben,
Fr. Josefstr. 18, 8700 Leoben, Austria}

\maketitle



\begin{abstract}
 Let $u$ be the solution to the following stochastic
evolution equation
\DEQSZ\label{erste}\lk\{\barray du(t,x)& = &A u(t,x)\: dt +
B\;\sigma(u(t,x)) \, dL(t),\quad t>0;
\\ u(0,x)& =&x\earray\rk.
\EEQSZ taking values in an Hilbert space $\HH$, where $L$ is a
$\RR$ valued L\'evy process, $A:H\to H$ an infinitesimal generator
of a strongly continuous semigroup, $\sigma:H\to \RR$ bounded from
below and Lipschitz continuous, and $B:\RR\to H$ a possible
unbounded operator. A typical example of such an equation is a
stochastic Partial differential equation with boundary L\'evy noise. Let
$\CP=(\CP_t)_{t\ge 0}$
 the corresponding Markovian semigroup.

We show that, if the system \DEQSZ\label{erste-cont}\lk\{\barray
du(t)& = &A u(t)\: dt + B\;v(t),\quad t>0;
\\ u(0)& =&x\earray\rk.
\EEQSZ is approximate controllable in time $T>0$, then under some additional conditions on $B$ and $A$,
for any $x\in H$ the probability measure $\CP_T^\star \delta_x$ is
positive on open sets of $H$. Secondly, as an application, we investigate under which condition on 
the L\'evy process $L$ and on the operator $A$ and $B$ the
solution of Equation \eqref{erste}  is asymptotically strong Feller, respective, has a unique invariant measure.
 We apply  these results to
the damped wave equation driven by L\'evy boundary noise.
\del{  and show
that
 the corresponding Markovian semigroup admits a
unique invariant measure.}
\del{, i.e. for any $t>0$ and $\phi\in C ^1(\HH)$ there we
have
\DEQS \la D P_t \phi (x),a\ra \le C_1 t ^{1-p} |\phi|+
C_2|\phi'|, \quad x\in \HH, a\in \HH. \EEQS}
\end{abstract}

\section{Introduction}

{
To present the aim of this paper,  let $\HH$ 
be 
a  Hilbert space. Let $u$ be the unique solution
of the infinite dimensional system with Poissonian noise, formally
written as
\DEQSZ\label{erste1}
\lk\{ \begin{array}{rcl} du(t,x) &=& Au(t,x) \: dt +\int_\CH {B}\, \sigma(u(t))\, z\; \tilde \eta (dz,dt),\quad t>0,\\
u(0,x)&=&x.
\end{array}\rk.
\EEQSZ
In this equation, $A:\HH\to\HH$ is a linear operator generating a strongly continuous semigroup on $\HH$,
 $B:\CH \to\HH$
 is a certain mappings specified later, $\sigma:H\to \RR$ bounded from
below and Lipschitz continuous, and
$\eta:\CB(\CH )\times \CB(\RR ^ +)\to \NN_0\cup\{\infty\}$ is a compensated Poisson random measure over a 
probability space $\mathfrak{A}=(\Omega,\CF,(\CF_t)_{t\ge 0},\PP)$ and intensity measure $\nu$.
Let $\CP=(\CP_t)_{t\ge 0}$ 
be the Markovian semigroup induced on $\HH$, i.e.
\DEQSZ\label{estimate} \CP_t \phi(x) := \EE \phi(u(t,x)),\quad
x\in\HH,\quad  t>0,\quad  \phi\in C(\HH). \EEQSZ } A typical
example of such an equation is a stochastic partial differential
equation with boundary noise.
The aim of this paper is to verify  under which conditions on $A$,
$B$ and $\eta$ the Markovian semigroup generated by the solution
of \eqref{erste1} is irreducible and admits a unique invariant
measure.

Regularity properties of the Markovian semigroups of stochastic
processes play an important role in studying the long time
behavior of the process. Concerning the uniqueness of the
invariant measure of SPDEs driven by L\'evy processes some results
exist. One of the first results in this direction were established
in the articles of Chojnowska-Michalik \cite{MR798303,chon}. Next,
Fournier \cite{MR1751168} investigated SPDEs driven by space time
Poissonian noise. \del{R\"ockner and Wang considered in
\cite{MR1996872} the approached by the generalised Mehler
semigroup and investigated also question of similar kind, but
since we are looking for SPDEs with purely jump processes, the
results do not coincides.} Applebaum analysed in \cite{05126807}
the analytic property of the generalised Mehler semigroup induced
by L\'evy noise and in \cite{app} the self-decomposability of a
L\'evy noise in Hilbert space. Further works are the two articles
of Priola and Zabczyk \cite{priola-1,priola-2}. We also refer to
\cite{Szarek}, \cite{Priola-4}, \cite{Priola-5} for some recent
results and review of progress for the study of the ergodicity of
the Markovian semigroup associated to the solution of a L\'evy
driven SPDEs. {The proofs of the results in \cite{priola-2,
Priola-4, Priola-5} rely  on the cylindrical and
$\alpha$-stability of the noise, hence their approach does not
cover the case we are treating in this paper. { In the present
work we show that if the system \eqref{erste-cont} is null
controllable, then the Markovian semigroup of solutions to
\eqref{erste} is irreducible. We applied our result to stochastic
evolution equation with L\'evy noise boundary conditions. For
results related to SPDEs with white-noise boundary condition we
refer to \cite{DZ92},\cite{Maslowski}, \cite{ZBrzezniak+SPeszat}.
For stochastic evolution equation driven by Wiener noise a similar
result was established long ago. Indeed the Markovian semigroup of
an Ornstein-Uhlenbeck is irreducible and strong Feller if
\eqref{erste-cont} is null controllable. For this result we refer
to the books of Da Prato and Zabczyk \cite{DZ92} and \cite{DZ97}
and references therein. We also note that in the present paper we
prove the uniqueness of invariant measure for the Markovian
semigroup of solution to \eqref{erste} if a certain notion of null
controllability is satisfied by \eqref{erste-cont}. In fact if
\eqref{erste-cont} is null controllable with vanishing energy (see
Section \ref{sec:3} for the definition), then we are able to show
that the Markovian semigroup of \eqref{erste} satisfies the
asymptotic strong Feller property. The irreducibility and the
asymptotic strong Feller property which is introduced by Hairer
and Mattingly in \cite{mat-hai-2006} will imply the uniqueness of
invariant measure. For SPDEs driven by L\'evy noise it is proved
in \cite{MR1996872} that the null controllability implies the
strong Feller property of the solution to Ornstein-Uhlenbeck
system driven by L\'evy noise with non-zero Gaussian part (see
\cite[Corollary 1.2]{MR1996872}). Unfortunately the result in
\cite{MR1996872} tells us nothing about the property of the
Markovian semigroup when we consider an Ornstein-Uhlenbeck driven
by pure jump noise. Hence our work is an extension of the results
in \cite{DZ92}, \cite{DZ97} and \cite{MR1996872}, in the sense
that we can prove uniqueness of invariant measure for SPDEs driven
by multiplicative and pure jump noise.}

\del{ Picard \cite{MR1402654,MR1449834} obtained a criterion for
existence and regularity of the density without using the Girsanov
transformation based on a duality formula.}
\par
\del{E.g.\ if the semigroup is strong Feller, then one
characterize the asymptotic behaviour of the semigroup, in
particular, one can show that if there exists a invariant measure,
this measure is unique.} \del{If the Markovian semigroup is strong
Feller and satisfies an irreducibility property, then it admits a
unique invariant measure. {To relax these conditions Hairer and
Mattingly introduced in \cite{mat-hai-2006} the so called
asymptotic strong Feller property.} In particular, they proved
that for the uniqueness of the invariant measure it is sufficient
the existence of the invariant measure, some nondengeneracy
property
 and that the Markovian semigroup is asymptotically strong
 Feller.}

\par
{The structure of the paper is the following. In Section
\ref{sec:2}  we give the hypotheses used throughout the paper and
prove an important relation between the irreducibility property
and approximate null controllability. Roughly speaking we could
prove in Section \ref{sec:2} that any ball centered at  the origin
( resp., at any $x \in H$) has positive measure if
\eqref{erste-cont} is approximate (resp., exactly) null
controllable. Section \ref{sec:3} is devoted to the proof of the
uniqueness of the invariant measure of the Markovian semigroup
associated to the solution of \eqref{erste1}. In fact, we
established that the Markovian semigroup satisfies the asymptotic
strong Feller property if \eqref{erste-cont} is null controllable
with vanishing energy. The asymptotic strong Feller and the
irreducibility of the aforementioned semigroup implies the
uniqueness of the invariant measure. We apply our results in
Section \ref{sec:4} to a damped wave equations driven by boundary
L\'evy noise. The last part of the paper is some appendices
collecting some technical results about the change of measure.}
The proofs of our results are a combination of the change of
measure formula given by Bismuth, Graveraux and Jacod
\cite{MR1008471}
  and Sato \cite{sato-book} (see also \cite{abso}) and the method used by Maslowski and Seidler \cite{bohdan}. 

\begin{Notation}
Let $\RR^+:=(0,\infty)$, $\RR^+_0:=(0,\infty)$,
$\NN_0:=\NN\cup\{0\}$ and $\bar\NN := \NN_0\cup\{\infty\}$. Let
$(Z,\CZ)$ be a measurable space. By $M_+( Z)$ we denote the family
of all positive measures on $Z$, by $\CM_+(Z)$ we denote the
$\sigma$-field on $M_+(Z)$ generated by functions $i_B:M_+(Z)
\ni\mu \mapsto \mu(B)\in \RR_+$, $B\in \CZ$. By $M_I( Z)$ we
denote the family of all $\sigma$--finite integer valued measures
on $Z$, by $\CM_I(Z)$ we denote the $\sigma$-field on $M_I(Z)$
generated by functions $i_B:M_I(Z) \ni\mu \mapsto \mu(B)\in
\bar\NN$, $B\in \CZ$. By $M_\sigma ^ +(Z)$ we denote the set of
all $\sigma$--finite and positive measures on $Z$,
 by $\CM ^ +_\sigma (Z)$ we denote the $\sigma$-field on $M ^ +_\sigma(Z)$
generated by functions $i_B:M ^ +_\sigma(Z) \ni\mu \mapsto
\mu(B)\in \RR$, $B\in \CZ$. We denote by $B(Z)$ the set of all
Borel measurable, real-valued, bounded functions.

For a Hilbert space $H$, by
  $C_b(H)$ the space of all uniformly continuous and bounded mappings $\phi:H\to \RR$ endowed with the norm $|\phi|_\infty= \sup_{x\in H} |\phi(x)|$.
\del{  and by $C^{(1)}_b(H)$ the set of all differentiable
functions with uniformly bounded derivative.}
\del{For any Banach space $Y$ and number $q\in[1,\infty)$, we
denote by $\mathcal{N}(\RR_+;Y)$
 the space (of equivalence classes) of
{progressively-measurable} processes $\xi : \rrint\times\Omega\to Y$ and by 
$\mathcal{M}^q ( \rrint;Y)$ the Banach space consisting of those
$\xi \in \mathcal{N}( \rrint;Y)$ for which $
\mathbb{E}\int_0^\infty \vert\xi (t)\vert^q_Y\,dt<\infty$.}

\del{Let $A$ be the generator of an analytic semigroup on $X$.
We define the extrapolation spaces of $A$ conform \cite[Section
2.6]{Pazy:83}; i.e.\ for $\alpha>0$ and $\lambda \in \CC$ such
that $\frak{R}e(\lambda)>\omega$ we define $H_{-\alpha}$ to be the
closure of $H$ under the norm $\| \cdot \|_{-\alpha} := \|
(\lambda I -A)^{-\alpha} \cdot \|$. We also define the fractional
domain spaces of $A$, i.e.\ for $\alpha>0$ we define
$H_{\alpha}=D((\lambda I -A)^{\alpha})$ and $\| \cdot\|_{\alpha}
:= \| (\lambda I -A)^{\alpha} \cdot \|$. One may check that
regardless of the choice of $\lambda$ the extrapolation spaces and
the fractional domain spaces are uniquely determined up to
isomorphisms: for $\alpha>0$ one has $(\lambda I -A)^{\alpha}(\mu
I -A)^{-\alpha}\in \CL(H)$. \del{ and:
$$\| (\lambda I -A)^{\alpha}(\mu I -A)^{-\alpha}\|_{L(H)}\leq C(\omega,\theta,K,\lambda,\mu),$$
where $C(\omega,\theta,K,\lambda,\mu)$ denotes a constant
depending only on $\omega,\theta,K,\lambda,$ and $\mu$. Moreover,
for $\alpha,\, \beta\in \RR$ one has $(\lambda
I-A)^{\alpha}(\lambda I-A)^{\beta}=(\lambda I-A)^{\alpha+\beta}$
on $X^A_{\gamma}$, where $\gamma= \max\{\beta,\alpha+\beta\}$ (see
\cite[Theorem 2.6.8]{Pazy:83}).}}
\end{Notation}

\section{Irreducibility of the Markovian semigroup associated to the equation \eqref{erste}}\label{sec:2}

One way to handle L\'evy processes is to work with the associated Poisson random measure.
In this section we will define the setting in which the results can be formulated.
We start with defining a time homogenous Poisson random measure.
\begin{definition}\label{def-Prm} 
Let $(Z,\CZ)$ be a measurable space and let $(\Omega,\CF,\FF,\PP)$
be a filtered probability space with right continuous filtration
$\FF=(\CF_t)_{t\ge 0}$. \noindent A {\sl time homogeneous Poisson
random measure} $\eta$ on $(Z,\CZ)$ over $ 
(\Omega,\CF,\FF,\PP)$,
is a measurable function $\eta: (\Omega,\CF)\to ( M_I(Z\times
\rrint), \CM_I(Z\times \rrint))$, such that
\begin{enumerate}
\item[(i)] $\eta(\emptyset\times I)=0$ a.s.\ for $I\in\CB([0,\infty) )$
and $\eta(A\times \emptyset)=0$ a.s.\ for $A\in\CZ$;
\item[(ii)] for each $B\times I \in \CZ \times \CB( \rrint)$,
 $\eta(B\times I):=i_{B\times I}\circ \eta : \Omega\to \bar{\mathbb{N}} $
 is a Poisson random variable with parameter\footnote{If $\nu(B)\lambda(I) = \infty$, then obviously $\eta(B\times I)=\infty$ a.s..} $\nu(B)\lambda(I)$.
\item[(iii)] $\eta$ is independently scattered, i.e.\ if the sets
$ B_j\times I_j \in \CZ\times \CB( \rrint)$, $j=1,\cdots, n$,
are pairwise disjoint, then the random variables $\eta(B_j\times
I_j )$, $j=1,\cdots,n $ are mutually independent.
\item[(iv)] for each $U\in \CZ$, the $\bar{\mathbb{N}}$-valued
process $(N(t,U))_{t>0}$ defined by
$$N(t,U):= \eta( U\times(0,t]), \;\; t>0$$
is $(\mathcal{F}_t)_{t\geq 0}$-adapted and its increments are
independent of the past, i.e.\ if $t>s\geq 0$, then
$N(t,U)-N(s,U)=\eta((s,t]\times U)$ is independent of
$\mathcal{F}_s$.
\end{enumerate}
The measure $\nu$ defined by
$$
\nu: \CZ\ni A \mapsto \EE\eta(A\times(0,1])\in \bar \NN
$$
is called the intensity of $\eta$.
\end{definition}

If the intensity of a Poisson random measure is a L\'evy measure, then one can construct from the Poisson random measure a L\'evy process. Vice versa,
tracing the jumps, one can find a Poisson random measure associated to each L\'evy process.
For more details on this relationship we refer to \cite{applebaum,maxreg}.
%
%
%
%

Let $\mathfrak{A}=(\Omega,\CF,\FF,\PP)$ be a complete probability measure with
right continuous filtration $\FF=(\CF_t)_{\{t\ge0\}}$, $\eta$ be a
time homogeneous Poisson random measure on 
$\ZZ$
over $\mathfrak{A}$ with intensity $\nu$ being a L\'evy measure\footnote{A L\'evy measure on $\CH$ is a $\sigma$--finite measure such that
$\nu(\{0\})=0$ and $\int_\CH (|z|^2 \wedge 1) \nu(dz)<\infty$.} and
compensator $\gamma$ defined by
$$
\gamma:\CB(\ZZ)\times \CB(\tim )\ni (A\times I) \mapsto
\gamma(A\times I):=\nu(A)\; \lambda(I)\in \RR ^+_0.
$$

\del{Let $\CZ$ be a Banach space and $\CB_\CZ$ be the unit ball in
$\CZ$, i.e.\ $\CB_\CZ:= \{ z\in\CZ: |z|_\CZ=1\}$. }
\del{\begin{hypo}\label{nondeg} There exists
 a measurable function
$k: \RR \to \RR^+$ such that
$$\nu(A)= \int_{\RR} 1_A(rz)\, k(z,r) \,dr\, \sigma(dz),
\quad A\in\CB(\CH).
$$
\end{hypo}
}

\begin{hypo}\label{nondeg1}\label{hypo3}
We assume that the L\'evy measure has a density $k$ and there
exist an index $\alpha\in(1,2]$ and constants $K_0>0$ and $r_0>0$
such that
$$
k(r) = K_0 |r|^{-\alpha-1},\quad \mbox{for all} \quad |r|\ge r_0.
$$
\end{hypo}

\del{\begin{definition}\label{def:levy}(see Linde \cite[Chapter 5.4]{0665.60005})
Let $E$ be a separable Banach space and let $E^\prime$ be its
dual. A symmetric\footnote{i.e. $\lambda(A)=\lambda(-A)$ for all
$A\in\CB(E)$,} $\sigma$-finite Borel measure $\lambda$ on $E$ is
called a {\sl L\'evy measure} if and only if
\begin{trivlist}
\item[(i)]
 $\lambda(\{0\} )=0$, and
\item[(ii)]
 the
function\footnote{As remarked in Linde \cite[Chapter
5.4]{0665.60005} we do not need to suppose that the integral $\int_E
(\cos\langle x,a\rangle -1) \; \lambda(dx)$ is finite. However, see
Corollary 5.4.2 in ibidem, if $\lambda$ is a symmetric L\'evy
measure, then, for each $a \in E^\prime$, the integral in
question is finite. }
$$
E^\prime \ni a\mapsto \exp \lk( \int_E (\cos\langle x,a\rangle -1)
\; \lambda(dx)\rk)
$$
is a characteristic function of a Radon measure on $E$.
\end{trivlist}
An arbitrary $\sigma$-finite Borel measure $\lambda$ on $E$ is
called a L\'evy measure provided its symmetric part
$\frac12(\lambda+\lambda^-)$, where $\lambda^-(A):=\lambda(-A)$,
$A\in\CB(E)$, is a
L\'evy measure.
The class of all L\'evy measures on $(E,\CB(E))$ will be denoted by
$\CL (E)$.
\end{definition}
\begin{remark}
If $E$ is a Hilbert space, then the set of all L\'evy measures $\CL(E)$ are all $\sigma$--finite measure $\nu$ such that we have
$\nu(\{0\})=0$ and
$$
\int_E \lk(|z|^2 \wedge 1\rk)\nu(dz)<\infty.
$$
\end{remark}}

\del{\begin{hypo}\label{antisym}
We assume that the function $k$ is symmetric in the following sense. For all
$z\in\CB_\CZ$ and $r\ge 0$ we have $k(z,r)=k(-z,r)$.
\end{hypo}
}

\del{Let $\eta$ be a Poisson random measure on $\CH $ with
intensity measure $\nu$. Let us assume that the Hypothesis
\ref{nondeg1} is valid.
\begin{remark}
Hypothesis \label{antisym} implies that the corresponding L\'evy measure
$\nu$ is symmetric, i.e. $\nu(A)=\nu(-A)$, $A\in\CB(\CZ)$.
\end{remark}}

From here and throughout the rest of the paper, let us assume that   $H$  a  is
Hilbert space, $A:H\to H$  a  generator of a strongly continuous semigroup $(e^{-tA})_{t\ge 0}$
on $H$
and $B:\CH\to D(A^{-\gamma})$ is bounded for some $\gamma<\frac 12$. Also let $\sigma:H\rightarrow \mathbb{R}$ be a Lipschitz mapping satisfying
$$C_\sigma< |\sigma(u)|\le \ell (1+|u|), $$ for some positive constants $C_\sigma$, $\ell$ and for any $u\in H$.
Let $u$ be the solution of the following stochastic evolution equation 
\DEQSZ\label{eqn-00}
\lk\{ \barray du(t,x) &=& Au(t,x) + \int_\CH B \sigma( u(t,x))\, z
\tilde \eta(dz,ds), \\
u(0,x) &=& x\in H.
\earray\rk.
\EEQSZ
Typical examples of such system are SPDEs with boundary noise and are presented in the following examples
(for more details we refer to section \ref{dampedwave}).
\begin{example}\label{damped-ex}
We consider the vibration of a string of length $2 \pi$ where one
end is fixed and the other end is perturbed by a Levy noise. To be
more precise, let
 $T>0$ and $\alpha>0$. We consider the system
\DEQSZ\label{damped}
\\
\nonumber
\lk\{ \barray
u_{tt}(t,\xi) &-& u_{\xi\xi}(t,\xi) + \alpha u(t,\xi) = 0,\quad t\in (0,T),\, \xi\in (0,2 \pi ),
\\
u(t,0)&=& 0 ,\quad t\in (0,T),
\\ u_\xi(t,2 \pi ) &=& \log\lk( 2+ |u(t)|_{L^ 2 (\CO)}\rk) \, \dot{L}_{t},\quad t\in (0,T), 
\\ u(0,\xi)&=& x_0(\xi), \quad  
u_t(0,\xi)=x_1(\xi), \quad  \xi\in (0,2 \pi ),
\earray\rk.
\EEQSZ where $\dot{L}$ is the Radon Nikodym derivative of a real
valued  L\'evy process with intensity  measure $\nu$, $x_0\in
H_0^1(0,2 \pi )$ and $x_1\in L^2(0,2 \pi )$.
\end{example}

\begin{example}\label{heat-ex}
We consider a one--dimensional rod $(0,1)$. A L\'evy noise  is added  at the boundary $\xi=1$, while the boundary $\xi=0$ is assumed to be perfectly isolated.
 To be
more precise, let
 $T>0$. We consider the system
\DEQSZ\label{heat}
\\
\nonumber
\lk\{ \barray
u_{t}(t,\xi) &-& u_{\xi\xi}(t,\xi)  = 0,\quad t\in (0,T),\, \xi\in (0,1 ),
\\
u_\xi(t,0)&=& 0 ,\quad t\in (0,T),
\\ u_\xi(t,1 ) &=&  \dot{L}_{t},\quad t\in (0,T), 
\\ u(0,\xi)&=& x_0(\xi), \quad  
  \xi\in (0,1 ).
\earray\rk.
\EEQSZ Here  $\dot{L}$ is the Radon Nikodym derivative of a real
valued  L\'evy process with intensity  measure $\nu$, $x_0\in
L^2(0,1)$ .
\end{example}
The existence of solution to the stochastic equations in these
examples can be established by fixed point argument as used in
\cite{PESZAT+ZABZCYK} and \cite{Maslowski}. \del{ Let $u$ be the
solution of the stochastic evolution equation
\DEQSZ \label{eq-0} \lk\{ \barray du(t) &=& A u(t)\, dt + \int_\CZ
B\
, z \tilde \eta(dz,dt)
\\
u(0) &=& u_0. \earray \rk.
 \EEQSZ}

\del{}

If $\int_\CH |z|^2 \, \nu(dz)<\infty$, then 
the Markovian semigroup $(\CP_t)_{t\ge 0}$ defined by
\DEQS 
\CP_t\phi(x) := \EE \phi( u(x,t)),\quad \phi\in C_b(\HH),\; t\ge
0, \EEQS is a stochastically continuous Feller semigroup on $C_b(H)$. That is $(\CP_t)_{t\ge 0}$
satisfies (see \cite{MR2244975})
\begin{enumerate}
  \item $\CP_t\circ \CP_s = \CP_{t+s}$;
  \item for all $\phi\in C_b(H)$ and for all $x\in H$ we have $\lim_{t\to 0} \CP_t \phi(x) =
  \phi(x)$.
\end{enumerate}
Item (1) is clear. In order to verify (2) let $\phi\in C_b(H)$
with $|\phi|_\infty=1$. Item (2) follows by the fact that
$\lim_{t\to 0}\EE\phi(u(t,x))=\phi(x)$, or, for all $\ep>0$ there
exists a $\delta>0$ such that $\lk|
\EE\phi(u(t,x))-\phi(x)\rk|\le \ep$ for all $0\le t<\delta$. Fix
$\ep>0$. \del{First, since $\EE |u(t,x)|_{H_1}<\infty$ for
$t\in[0,1]$ and $H_1\hookrightarrow H$ compactly, there exists a
compact set $K_\ep$ such that $\PP\lk( u(t,x)\not\in K_\ep\rk) \le
\frac \ep 3$.} Since $\phi$ is uniformly continuous on $H$, there
exists a $\delta_1>0$ such that $|\phi(x)-\phi(y)|\le \frac \ep 2$
for all $x,y\in H$, $|x-y|_H\le \delta_1$. Then for $t\le
\delta:=\frac \ep 6 \delta_1^2 $ we know by the Chebyscheff
inequality that $\PP\lk(|u(t,x)-x|_H\ge \delta_1 \rk)\le \frac\ep
2$. Hence, \DEQS
&& \EE | \phi(u(t,x))-\phi(x)| \le 
\\
&& \quad {} \EE\lk[ \lk|\phi(u(t,x))-\phi(x)\rk|  \mid 
 | u(t,x)-x|\ge \delta_1\rk]\PP\lk( |u(t,x)-x|_H\ge \delta_1 \rk)
\\
&& \quad + \EE\lk[ \lk|\phi(u(t,x))-\phi(x)\rk|  \mid 
 | u(t,x)-x|< \delta_1\rk]\le \frac \ep 2 + \frac \ep 2 
.\EEQS
%
It follows that Markovian $(\CP_t)_{t\ge 0}$  on $C_b(H)$  is a stochastically continuous. 

\del{Under appropriate conditions on $B$, it can be shown that $u$ admits a Markovian transition semigroup $(\CP_t)_{t\ge 0}$ 
on $H$ given by 
\DEQSZ\label{def-mark} \CP_t\phi(x) := \EE
\phi( u(x,t)),\quad \phi\in C_b(\HH),\quad t\ge 0. \EEQSZ
}

Before continuing we would like to introduce some definitions from
control theory.
%
%
%
Again,  $H$ denotes a Hilbert space, $A:H\to H$  a  generator of a
strongly continuous semigroup $(e^{-tA})_{t\ge 0}$ on $H$ and
$B:\CH \to H$. Fix $T>0$. Then we say that the system
\DEQSZ\label{eq-0-c} \lk\{ \barray \dot u^c(t,x,v) &=& A \uc
(t,x,v)+ B v(t),\quad t\ge 0,
\\
\uc (0,x,v)&=& x, \earray\rk.
\EEQSZ
is {\em null controllable in time $T$}, iff for any $x\in H$ there
exists a $v\in L^2([0,T];\CH )$ 
such that $\uc (T,x,v)=0$. We say  that the system  \eqref{eq-0-c}
is {\em approximate null controllable in time $T$}, iff for any
$x\in H$ and $\ep>0$ there
exists a $v\in L^2([0,T];\CH )$ 
such that $|\uc (T,x,v)|_H\le \ep$.
We say that the system  \eqref{eq-0-c} is {\em controllable in
time $T$} in $x\in\CCH$ if for each $y\in \CCH$
and $\ep>0$ there exists 
a control $v\in L^2([0,T];\CH )$ such that $  \uc (T,x,v)=y$.
We say that the system  \eqref{eq-0-c} is {\em approximate
controllable in time $T$} in $x\in\CCH$ if for each $y\in \CCH$
and $\ep>0$ there exists 
a control $v\in L^2([0,T];\CH )$ such that
$
\lk| \uc (T,x,v)-y\rk| \le \ep.
$
\begin{remark}
\begin{itemize}
  \item The system \eqref{eq-0-c} associated to the wave equation with boundary control described in Example \eqref{damped-ex} is exactly controllable in time $T>2\pi$ (Zuazua \cite{zuazua});
  \item The system \eqref{eq-0-c} associated to the heat  equation with Neumann boundary control described in Example \eqref{heat-ex} is approximate controllable in time $T>0$ (Laroche, Martin and Rouchon \cite{laroche}).
\end{itemize}

\end{remark}

\del{By the linearity of $A$ and $B$ it follows that for all $\ep>0$ and $C>0$, %
there exists a constant $K>0$ such that for all $x\in
\CB_\CCH(C):=\{ x\in \CCH: |x|\le C\}$ there exists a $\delta>0$
and $v_0\in L^2([0,\infty);\CH )$, $\|v_0\|_{L^2([0,\infty);\CH
)}\le K$, such that for all $v\in L^2([0,\infty);\CH )$ with $\|
v_0-v\|_{v\in L^2([0,\infty);\CH )}\le \delta$ we have $ \lk| \uc
(t,x,v)\rk| \le \ep. $}

For all $C>0$ we set  $$\CR_H(C):=\{ z\in H: |z|\le C\}.$$
Let $u$ be the solution of the stochastic evolution equation
\DEQSZ \label{eq-0ohnev} \lk\{ \barray du(t,x) &=& A u(t,x)\, dt + \int_\CZ
B\, z \tilde \eta(dz,dt)
\\
u(0,x) &=& x. \earray \rk.
 \EEQSZ
 Then the following Theorem can be shown.
\begin{theorem}\label{nondeg2}
\del{Let $e^{-t A}$ be semigroup generated by $A$ and suppose that
\begin{equation}\label{SEM}
 \int_0^T |e^{-(T-s)A}B|^2 ds< \infty, \text{ for any } T>0.
\end{equation} Also, a}
Assume that the system \eqref{eq-0-c} is approximate null
controllable in time $T>0$ and that Hypothesis \ref{hypo3} is satisfied.
Let $u$ be a solution of Eq.  \eqref{eq-0ohnev}. Fix $x\in H$.
Then for any $\delta>0$ there exists  a
$\kappa>0$ such that
\DEQSZ\label{inv-22111} \PP\lk( u(T,x)\in
\CR_H(\delta)\rk)\ge \kappa. \EEQSZ
\end{theorem}

In case the system is exactly controllable the result of the above theorem can be strengthen as follows.

\begin{theorem}\label{nondeg22}
Assume that the system \eqref{eq-0-c} is  exactly  null
controllable in time $T>0$ and that Hypothesis \ref{hypo3} is satisfied.
Let $u$ be a solution of Eq.  \eqref{eq-0ohnev}.
Then for all $C>0$, for all $x\in \CR_H(C)$ and all $\delta>0$ there
exists  $\kappa>0$ such that
\DEQSZ\label{inv-221} \PP\lk( u(T,x)\in
\CR_H(\delta)\rk)\ge \kappa. \EEQSZ
\del{ where $\CR_H(\delta)=\{
z\in H: |z|< \delta\}$.}
\end{theorem}
\begin{remark}
{If  the system \eqref{eq-0-c} is controllable for a time $T$,
then one can replace the disk $\CR_H(C)$ centered at the point $0$
with a disk centered at any point $y\in H$.}
\end{remark}

In case the solution $u$ is \cadlag \ in $H$, the result can be strengthened as well.
Let $u$ be the solution of the stochastic evolution equation
\DEQSZ \label{eq-0mitv} \lk\{ \barray du(t,x) &=& A u(t,x)\, dt + \int_\CZ
B\, \sigma(u(t,x)) \, z \tilde \eta(dz,dt)
\\
u(0,x) &=& x, \earray \rk.
 \EEQSZ
 where $\sigma:H\to\RR$ is a Lipschitz mapping of linear growth and such that for certain  $C_\sigma>0$ we have  $|\sigma(x)|\ge C_\sigma$, $\forall x\in H$.

 Then the following two Theorems can be shown.
\begin{theorem}\label{nondeg2cad}
Assume that the system \eqref{eq-0-c} is approximate null
controllable in time $T>0$ and that Hypothesis \ref{hypo3} is satisfied.
Let $u$ be a solution of Eq.  \eqref{eq-0mitv}. Fix $x\in H$.
If $u$ is \cadlag \ in $H$, then
for any $\delta>0$ there exists  a
$\kappa>0$ such that \DEQSZ\label{inv-223} \PP\lk( u(T,x)\in
\CR_H(\delta)\rk)\ge \kappa. \EEQSZ
\end{theorem}
\begin{theorem}\label{nondeg2cad2}
Assume that the system \eqref{eq-0-c} is exactly null
controllable in time $T>0$ and that Hypothesis \ref{hypo3} is satisfied.
Let $u$ be a solution of Eq.  \eqref{eq-0mitv}. If $u$ is \cadlag \ in $H$, then
for all $C>0$, for all $x\in \CR_H(C)$ and all $\delta>0$ there
exists  $\kappa>0$ such that
\DEQSZ\label{inv-2211} \PP\lk( u(T,x)\in
\CR_H(\delta)\rk)\ge \kappa. \EEQSZ
\del{ where $\CR_H(\delta)=\{
z\in H: |z|< \delta\}$.}
\end{theorem}
\begin{example}
In Section \ref{sec:4} we will see that the linear problem \eqref{eq-0-c} for Example \ref{damped-ex} is exactly controllable and the solution is \cadlag \ in $L^2(\CO)$.
That means, the solution of system \eqref{damped} satisfies the assumptions of Theorem  \ref{nondeg2cad2}.
\end{example}
\begin{example}
 The linear system \eqref{eq-0-c} for Example \ref{heat-ex} is only approximately  controllable and the solution is not \cadlag \ in $L^2(\CO)$.
Thus the solution of system \eqref{heat} does not satisfy the assumptions of Theorem \ref{nondeg22}, Theorem \ref{nondeg2cad} nor Theorem \ref{nondeg2cad2}.
However, the assumptions of Theorem \ref{nondeg2} are satisfied.
\end{example}

\del{

\begin{remark}\label{nondeg22}
Assume that the system \eqref{eq-0-c} is  approximate null
controllable in time $T>0$ and that Hypothesis \ref{hypo3} is satisfied. Let
$y\in H$ be fixed.
Then for all $C>0$, for all $x\in \CR_H(C)$ and $\delta>0$ there
exists  $\kappa>0$ such that
\DEQSZ\label{inv-221} \PP\lk( u(T_\delta,x)\in
\CR_H(y,\delta)\rk)\ge \kappa, \EEQSZ where $\CR_H(y,\delta)=\{
z\in H: |z-y|< \delta\}$.
\end{remark}
}
%
%
%
%
\begin{proof}[{\bf Proof of Theorem \ref{nondeg2}}]
We will switch for technical reasons to another representation of
the Poisson random measure.
 Let
$\overline{\mathfrak{A}}=(\bar\Omega,\bar\CF,\bar\FF,\bar\PP)$ be
a filtered probability space with filtration
$\bar\FF=(\bar\CF_t)_{t\ge 0}$ and let $\mu$ be a Poisson random
measure on $\CH$ over $\overline{\mathfrak{A}}$ having intensity
measure $\lambda$ (Lebesgue measure).
The compensator of $\mu$ is denoted by $\gamma $ and given by
$$\CB(\CH)\times\CB([0,\infty))\ni A\times I\mapsto \gamma ( A\times I) := \lambda (A)\,\lambda(I).
$$
%
%
Let \DEQSZ\label{trans-cc} c:\RR ^+ \ni r \mapsto  \sup_{\rho>0}
\lk\{\int_\rho^\infty k(s)\, ds \ge r \rk\} \, &\mbox{ if } r>0.
\EEQSZ
\begin{remark}
Observe that Hypothesis \ref{nondeg1} 
implies that there exists a number $r_1>0$ and a constant
$\delta_0$ such that for all $r\ge r_1$
\end{remark}
 \DEQSZ\label{condition-derivative-c} {c(r)  }
= \delta_0 r ^ {-\frac1\alpha},\quad &\Aand & \quad {c(-r)
 } =-\delta_0 r ^ {-\frac 1\alpha}.
 \EEQSZ
A short calculation shows that the distributions of  $L=\{ L(t):0\le t<\infty\}$ and $L^c=\{ L^c(t):0\le t<\infty\}$ are equal, 
where 
$$
L (t) := \int_0 ^t \int_\CH z \,\tilde \eta(dz,ds) ,\quad t\ge 0,
$$
and 
$$
L ^c(t) := \int_0 ^t \int_\CH  c\lk(z\rk)\,\tilde \mu(dz,ds)
,\quad t\ge 0.
$$
Now, the stochastic evolution equation given in  \eqref{eq-0ohnev}
reads as follows \DEQSZ\label{eqn-00-mc} \lk\{ \barray du(t,x) &=&
Au(t,x) + \int_\CH B  \, c(z)
 \tilde \mu(dz,ds), \\
u(0,x) &=& x\in H. \earray\rk. \EEQSZ

\del{Let $u_{p}$ the solution to
 \DEQS \lk\{ \barray du_p(t,x) &=&
Au_p(t,x) + \int_\CH B  \sigma_p( t)\, c(z)
 \tilde \mu(dz,ds), \\
u_p(0,x) &=& x\in H, \earray\rk. \EEQS
where $\sigma_p$ is a real valued progressively measurable process. Then,
\DEQS
\EE

Since $u$ is $\bar \FB$--adapted, the process $\RR^+_0\ni t\mapsto \sigma(u(t))$ is progressively measureable and
there exists a predictable process $\sigma_p$ such that $\EE\int_0^T \lk|\sigma(u(s))-\sigma_p(s)\rk|^2\, ds\le \frac \delta 4$.
}

Fix $\delta>0$, $T>0$ and $x\in \CR_H(C)$. In order to prove Lemma
\ref{nondeg2} we need a result from control theory.
Given $v\in L^2([0,\infty);\CH )$, let  $\uc$ be the solution to
(see system \eqref{eq-0-c}) \DEQSZ\label{normal-system}\lk\{
\barray d\uc (t,x,v) &=&  A \uc (t,x,v)  dt + Bv(t)\, dt  , \quad
t\ge 0,
\\
\uc(0,x,v) &=& x. \earray\rk. \EEQSZ Since the system
\eqref{normal-system} is approximate null controllable,
%
there exists 
$v\in L^2([0,\infty);\CH )$  
such that 
\DEQSZ\label{spaeter} | \uc(T , x,v) |\le \frac \delta3 . \EEQSZ
%
%
Choose $R\ge r_1$ such that
$$
\lk(\frac 3\delta\rk)^2 \,C\,\TT  R^{1-\frac 2\alpha}\le \frac 12,
$$
\del{$$ \vardelta ^{-(1+\alpha +\alpha \beta)/\alpha} \lk[ \int_0
^{\TT } \lk| v(s)  \rk|\, ds + \TT \vardelta ^{1-\frac
1\alpha}\rk] +  C\, \vardelta ^{\frac 12-\frac 1\alpha}\le \frac
12
$$}
(here, $C$ is a generic constant, not depending on $\delta$, $\TT$
and $R$, see \eqref{hereC}) and put \DEQSZ\label{diff-app}
 g_\ww = \int_{\CR_\CH(\ww )}  c(z) \lambda (dz) 
 . 
\EEQSZ
Let $\theta:\Omega\times [0,\TT ]$ be a predictable transformation
of $\CH $ such that
$$
{v(s)}  + {g_\ww } = \int_{\CH \setminus  B_\CH(\vardelta)} \lk[ c(z)
- c(\theta(s,z))\rk] \lambda (dz),\quad s\in[0,\TT ].
$$
The existence of such a transformation is given by Lemma
\ref{trans_ex}. 
Let $\mu_\theta$ the following random measure defined by
$$
\mu_\theta:\CB(\CH)\times \CB([0,T])\ni(A\times I) \mapsto
  \int_{\RR^2}\int_I
1_A(\theta(s,z))\mu(dz,ds).
$$
Let $\QQ$ be the probability  measure on $\mathfrak{A}$ 
such that $\mu_\theta$ has compensator $\gamma $. Then, the
process $u_\mu^\theta $ defined by
\DEQSZ\label{eqn-0theta} \lk\{
\barray d u^\theta_\mu(t,x) &=& A u^\theta_\mu(t,x) dt+
\int_\CH B c(z)\, ( \mu_\theta -\gamma ) (dz,dt) \\
u^\theta_\mu(0,x) &=& x. \earray \rk. \EEQSZ has under $\QQ$ the
same law as $u$, in particular
$$
\EE^\PP \lk[ 1_{[0,\delta]}(|u(\TT ,x_0)|)\rk] = \EE^\QQ \lk[
1_{[0,\delta]}(|u^\theta_\mu(\TT ,x_0)|)\rk] .
$$
Due to Lemma \ref{density} the  density process
$\CG_\theta(t)={d\QQ^\theta_t\over d\bar \PP_t}$ satisfy the
following stochastic differential equation (see also page
\pageref{asdonebefroe})
\DEQS
 \lk\{\barray d\CG_\theta (t) &=& \CG_\theta(t)\lk(  \rho_z(  \kappa(|
v(s) |),z ) ) \sgn( v(s )) \rk)\,(\mu-\gamma )(dz,ds)
\\ \CG_\theta(0)&=& 1.
\earray\rk. \EEQS
Here $\rho$ is defined in \eqref{v0s} and $\kappa$ is the inverse of $\wt$ and defined on page \pageref{def-kappa}.
By  the choice of $\rho$, we know that $\CG$ is of finite variation and 
we obtain
 for $0\le t\le T$
\DEQS {\EE^{\bar \PP}}    \sup_{0\le s\le t} |\CG_\theta(s)|
  & \le &1 + {\EE^{\bar \PP}} \int_0 ^ t \int_{\RR^+}  \,|\CG _\theta
(s-)|   |\rho_z( \kappa (| v(s)| ),z) | \,dz\, ds
  , \EEQS
  respective, we obtain for $0\le r_1\le T$
\DEQS {\EE^{\bar \PP}}    \sup_{0\le s\le r_1} |\CG_\theta(s)|
  & \le &1 + {\EE^{\bar \PP}} \int_0 ^ {r_1} \int_{\RR^+}  \,|\CG _\theta
(s-)|   |\rho_z( \kappa (| v(s)| ),z) | \,dz\, ds
  \\
  &\le & 1 + {\EE^{\bar \PP}} \sup_{0\le s\le r_1} |\CG _\theta
(s)| \times
  \int_0 ^ {r_1} \int_{\RR^+}    |\rho_z( \kappa (| v(s)| ),z) | \,dz\, ds.
   \EEQS
By Corollary \ref{theta-ex} we have
\DEQS {\EE^{\bar \PP}}    \sup_{0\le s\le r_1} |\CG_\theta(s)|
  & \le & 1 + {\EE^{\bar \PP}} \sup_{0\le s\le r_1} |\CG _\theta
(s)| \times
  \int_0 ^ {r_1}    | v(s)|^2 \, ds.
   \EEQS
Now, if $ \int_0 ^ {r_1}    | v(s)|^2 \, ds\le \frac 12$, we get
\DEQS {\EE^{\bar \PP}}    \sup_{0\le s\le r_1} |\CG_\theta(s)|
  & \le & 1\times (\frac 12)^{-1}.
   \EEQS
\del{In  fact, one can even show  for $0\le t\le T$ 
\DEQS  {\EE^{\bar \PP}}  \sup_{r\le s\le t}| \CG _\theta (s)| &\le
& {\EE^{\bar \PP}}  \sup_{0\le s\le r}| \CG _\theta (s)| +
C(r_1)\, {\EE^{\bar \PP}} \lk[ \sup_{r\le s\le t}| \CG _\theta
(s)|\rk]
 \lk(  
 \int_r ^ t  |  v(s) |^2 \, ds \rk) .
\EEQS}
Since $v\in L^ 2([0,T];\RR)$, there exists a partition $\{t_j:
1\le j\le N\}$ of $[0,T]$ with
$$\int_{t_{j-1}}^{t_j}|v(s)|^ 2 \,
ds<\frac 1{2},\quad j=2,\ldots,N,
$$
and $N\le [2|v|_{L^2([0,T];\RR)}^2]+1$.
Therefore,  \DEQS  {\EE^{\bar \PP}}  \sup_{0\le s\le t_j}| \CG
_\theta (s)| &\le & \lk(\frac 12\rk) ^{-j},
\EEQS
and we can conclude that there exists a constant $C>0$
%
\DEQSZ\label{density11} {\EE^{\bar \PP}} \sup_{0\le s\le T}
|\CG_\theta(s) |&\le & C\, 2^{|v|^2_{L^2([0,T];\RR)}}  .
 \EEQSZ

On the other hand we know that under $\bar\PP$  the process
$u^\theta_\mu$ follows the following differential equation
\del{\DEQSZ\label{eqn-0c}
\lk\{ \barray    
d u^\theta_\mu(t,x) &=& A u^\theta_\mu(t,x) dt+ \int_\CH B\lk[ c(\theta(t,z))-c(z)\rk] \gamma (dz,dt)  
\\&&{} + \int_\CH  Bc(\theta(t,z)) ( \mu-\gamma ) (dz,dt),
\\ u^\theta_\mu(0,x) &=& x .
\earray\rk. \EEQSZ} \DEQSZ\label{eqn-0c}
\\
\nonumber
\lk\{ \barray    
d u^\theta_\mu(t,x) 
&=&A u^\theta_\mu(t,x) dt+ \int_\CH B\lk[ c(\theta(t,z))-c(z)\rk] (\mu-\gamma ) (dz,dt)  
\\&&{} +\int_\CH  B\lk[c(\theta(t,z))-c(z)\rk]\gamma(dz,dt)+
\int_\CH B c(z) ( \mu-\gamma ) (dz,dt),
\\ u^\theta_\mu(0,x) &=& x .
\earray\rk. \EEQSZ
Hence we can write
\DEQS
 \bar \PP\lk( |u(\TT ,x)|\le \delta\rk)=\QQ( |u^\theta_\mu(\TT ,x)|\le \delta) = \EE^\QQ \lk[1_{|u_\mu^\theta (\TT ,x)|\le \delta}\rk]
= \EE^{\bar \PP}\lk[   \CG_\theta(\TT ) 
1_{|u_\mu^\theta (\TT ,x)|\le \delta}\rk] . \EEQS \del{\coma{The
processes $u$ and $u^\theta_\mu$ have the same law, therefore the
above equation should be replaced by } \red{\DEQS
  \PP\lk( |u(\TT ,x)|\le \delta\rk)=\QQ( |u^\theta_\mu(\TT ,x)|\le \delta) = \EE^\QQ \lk[1_{|u_\mu^\theta (\TT ,x)|\le \delta}\rk]
= \EE^{\bar \PP}\lk[   \CG_\theta(\TT ) 
1_{|u_\mu^\theta (\TT ,x)|\le \delta}\rk] . \EEQS}}
By the inverse H\"older inequality we get  
\DEQS
 \bar {\PP}\lk( |u(\TT ,x)|\le \delta\rk)
 &\ge &{ \EE^{{\bar \PP}}\lk[1_{|u_\mu^\theta (\TT ,x)|\le \delta}\rk]\over \EE^{{\bar \PP}}\lk[  \CG_\theta(\TT ) \rk]}.
\EEQS
The denominator, i.e.\ $\EE^{\bar \PP}\lk[ |\CG_\theta(\TT )
|\rk]$, is bounded. In particular, we have by  \eqref{density11}
and Hypothesis \ref{hypo3} that
$$
\EE^{\bar \PP}\lk[  |\CG_\theta(\TT ) |\rk]
 \le C'\exp\lk(r_1^ { \alpha \beta_1+1\over \alpha} \int_0^T |v(s)|^2 ds\rk).
$$
Next, we handle the numerator. Observe that by \eqref{spaeter}
\DEQSZ\nonumber
\lqq{ \EE^{\bar \PP}\lk[1_{|u_\mu^\theta (\TT ,x)|\le \delta}\rk] }&&\\
&\ge &{\EE^{\bar\PP}}\lk[1_{|\uc (\TT ,x,v)|\le \frac \delta
3}\,\,1_{|\uc (\TT ,x,v)- u_{\mu}^\theta (\TT ,x)   |\le \frac
\delta3} \rk] = \EE^{\bar \PP}\lk[1_{|\uc (\TT ,x,v)-
u_{\mu}^\theta (\TT ,x)   |\le \frac \delta3} \,\rk]
.\label{difference-1} \EEQSZ
Rewriting the difference $\Delta(\TT )=
u_{\mu}^\theta (\TT ,x) -\uc (\TT ,x,v) $ as follows \DEQS \lqq{ \Delta(\TT ) =
\int_0^{\TT }  \int_\CH  e^{-(t-s)A}B \lk[ c(\theta(t,z))- c(z)
\rk]  (\mu-\gamma ) (dz,ds)}
&&\\
&&{} + \int_0^{\TT }  \int_{\CH} e^{-(t-s)A}B c(z)\, (\mu-\gamma )
(dz,ds)
 + \int_0^{\TT }e^{-(t-s)A}Bg_R \, ds
\\
&=&
 \int_0^{\TT }  \int_\CH  e^{-(t-s)A}B c(\theta(t,z))  (\mu-\gamma ) (dz,ds)\\
&&{}
 + \int_0^{\TT }e^{-(t-s)A}Bg_R \, ds.
\EEQS
we have \DEQSZ\nonumber \EE^{\bar \PP}\lk[1_{|u_\mu^\theta (\TT
,x)|\le \delta}\rk] &\ge &
 \EE^{\bar \PP}\lk[1_{|\Delta(\TT ) 
    |\le \frac \delta3}
\,\rk]. \label{difference_21} \EEQSZ
To give a lower  estimate of   $\bar \PP\lk( |\Delta(\TT )|\le
\frac \delta3\rk)$ we apply the Bayes Theorem and get
\DEQS && \bar \PP\lk( |\Delta(\TT )|\le \frac \delta3\rk)= \bar
\PP \lk( \mu(   B_\CH(\ww ) \times [0,\TT ])=0\rk)
\\ &\times &\bar \PP
\lk(  \lk|\int_0^{T_\vardelta} \, \int_{ \CH\setminus \CR_\CH(\ww
)} e^{-(t-s)A}Bc(\theta(s,z))(\mu-\gamma ) (dz,ds) \rk.\rk.
\\
&&+\underbrace{ \int_0^{T_\vardelta} \, \int_{  B_\CH(\ww )}
e^{-(t-s)A}B c(\theta(s,z)) (\mu-\gamma )(dz,ds)} _{=
\int_0^{T_\vardelta}   e^{-(t-s)A}B g_R \, ds}
\\ && - \int_0^{T_\vardelta}   e^{-(t-s)A}B g_R \, ds
\Bigg|\le {\delta\over 3 }\,\, \,\,\Big|\,\, \mu( \CR_\CH (\ww )
\times [0,\TT ])=0\Bigg)
\\
&+& \bar  \PP \lk( \mu( \CR_\CH(\ww ) \times [0,\TT ])>0\rk)
\\ &\times &\bar \PP
\lk(  \lk|\int_0^{T_\vardelta} \, \int_{\CH} e^{-(t-s)A}B
c(\theta(s,z))(\mu-\gamma )(dz,ds) \rk.\rk.
\\ &&\lk.\lk. -  \int_0^{T_\vardelta}   e^{-(t-s)A}B g_R  \, ds
\rk|\le {\delta\over 3 }\,\,\Big|\,\, \mu(  \CR_\CH (\ww ) \times
[0,\TT ])>0\rk)
\\
&\ge & \bar  \PP \lk( \mu(   \CR_\CH(\ww ) \times [0,\TT ])=0\rk)
\\ &\times &\bar \PP
\Bigg(  \lk|\int_0^{T_\vardelta} \, \int_{ \CH\setminus
\CR_\CH(\ww )} e^{-(t-s)A}B c(\theta(s,z))(\mu-\gamma )(dz,ds)\rk|
\le {\delta\over 3 }\,\,\Big|\,\, \mu(  \CR_\CH (\ww ) \times
[0,\TT ])=0\Bigg) . \EEQS
By the Chebyscheff inequality we know that
\DEQS && \bar \PP \lk(\lk| \int_0^{T_\vardelta} \, \int_{
\CH\setminus \CR_\CH(\ww )} e^{-(t-s)A}B c(\theta(s,z))
(\mu-\gamma )(dz,ds) \rk|\ge {\delta\over 3}\,\,\Big|\,\, \mu(
\CR_\CH(\ww ) \times [0,\TT ])=0\rk)
\\ \lqq{ \le \lk(\frac 3\delta\rk)^2 }
&&\\&\times &
 \EE^{\bar \PP}\lk[  \lk| \int_0^{T_\vardelta} \, \int_{ \CH\setminus \CR_\CH(\ww )} e^{-(t-s)A}Bc(\theta(s,z))(\mu-\gamma )(dz,ds)
\rk|^2\,\,\Big|\,\, \mu( \CR_\CH (\ww ) \times [0,\TT ])=0\rk].
\EEQS
Note that due to the fact that the random variables $\mu(
\CR_\CH(\ww ) \times [0,\TT ])$ and $\mu( A \times [0,\TT ])$ for
all $A\in \CB(\CH\setminus \CR_\CH(\ww ))$ are independent, we get
\DEQS &&\EE^{\bar \PP}\lk[  \lk| \int_0^{T_\vardelta} \, \int_{
\CH\setminus \CR_\CH(\ww )} e^{-(t-s)A}Bc(\theta(s,z))(\mu-\gamma
)(dz,ds) \rk|\,\,\Big|^2 \,\, \mu(  B_\CH(\ww ) \times [0,\TT
])=0\rk]
\\ &=&
\EE ^{\bar \PP}   \lk| \int_0^{T_\vardelta} \, \int_{ \CH\setminus
\CR_\CH(\ww )} e^{-(t-s)A}B c(\theta(s,z))(\mu-\gamma )(dz,ds)
\,\rk| ^2 . \EEQS
The Burkholder inequality and the fact that $c(\theta(t,z))\le c(z)$ give 
\DEQS && \EE^{\bar \PP}\lk[  \lk|\int_0^{\TT } \, \int_{
\CH\setminus \CR_\CH(\ww )} e^{-(t-s)A}Bc(\theta(s,z))(\mu-\gamma
)(dz,ds)\rk|^2 \rk]
\\
&\le & \EE^{\bar \PP}\lk[  \int_0^{\TT } \, \int_{ \CH\setminus
\CR_\CH(\ww )} \lk| e^{-(t-s)A}Bc(\theta(s,z))\rk|^ 2 \lambda \,
ds\rk]
\\
&\le & \EE^{\bar \PP}\lk[  \int_0^{\TT } \, \int_{ \CH\setminus
\CR_\CH(\ww )} \lk| e^{-(t-s)A}Bc(z)\rk|^ 2 \lambda \, ds\rk] \le
C\,\TT  \ww  ^{1-\frac 2\alpha}.
 \EEQS
 Therefore, collecting all together
  \DEQS \EE^{\bar \PP}\lk[1_{|\uc (\TT
,x,v)- u_{\mu}^\theta (\TT ,x)   |\le \frac \delta3}\rk]&\ge&
\frac{1}{C(R)}\lk(1-\lk(\frac 3\delta\rk)^2 \,C\,\TT  \ww  ^{1-\frac
2\alpha}\rk). \EEQS
Since $R$ is chosen in such a way that \DEQSZ\label{hereC}
\lk(\frac 3\delta\rk)^2 \,C  \,\TT \ww ^{1-\frac 2\alpha}\le \frac
12 \EEQSZ and using the fact that
$$
\bar \PP \lk( \mu(  \CR_\CH(\ww )  \times [0,\TT ])=0\rk) =
e^{-\lambda (\CR_\CH(\ww ))\TT  }
$$
we get
\DEQS 
\lqq{ \EE^{\bar \PP}\lk[1_{|u_\mu^\theta (\TT ,x)|\le \delta}\rk]
\ge \EE^{\bar \PP}\lk[1_{|\Delta(\TT )   |\le \frac \delta3}
\,\rk]\ge \bar \PP \lk( \mu(  \CR_\CH(\ww )  \times [0,\TT
])=0\rk)
}&&\\
 &&\times \Bigg( 1-
\bar \PP \lk(  \lk|\int_0^{\TT } \, \int_{\CH\setminus \CR_\CH(\ww
)} e^{-(t-s)A}Bc(\theta(s,z)) (\mu-\gamma )(dz,ds) \rk|\ge
{\delta\over 3}\,\,\Big|\,\, \mu(  \CR_\CH(\ww ) \times [0,\TT
])=0\rk)
 \Bigg)
 \\
 &\ge & e^{-\lambda (\CR_\CH(\ww ))\TT  } C(R)\lk(  1-\lk(\frac 3\delta\rk)^2 \,C\, \TT  \ww  ^{1-\frac 2\alpha}\rk)
 \ge \frac 12 \,\, e^{-\lambda (\CR_\CH(\ww ))\TT  } C(R).
 \EEQS
Hence, 
 we have shown that
\DEQSZ\label{ending}
 \bar {\PP}\lk( |u(\TT ,x)|\le \delta\rk) \ge { \frac 12 \,\, e^{-\lambda (\CR_\CH(\ww ))\TT  } } C(R) \exp(-|v|_{L^2 ([0,T];\RR)}^2 ),
\EEQSZ
which gives the assertion.
\end{proof} 

\begin{proof}[{\bf Proof of Theorem \ref{nondeg22}}]
%
%
Let $x\in\CB_H(C)$ and  $u^c$ the solution to
\DEQSZ\label{eq-0-c-1} \lk\{ \barray \dot u^c(t,x,v) &=& A \uc
(t,x,v_x)+ B v_x(t),\quad t\ge 0,
\\
\uc (0,x,v_x)&=& x, \earray\rk. \EEQSZ
In order to show Theorem \eqref{nondeg22}, we have to show that the RHS of Inequality \eqref{ending}
can be estimated from below for all $x\in\CB_H(\delta)$.
That is, we have to show that for any $x\in\CR
_H(C)$
there exists a control $v_x\in L^2([0,T];\RR)$ and a constant $K>0$  such that
$u^c(T,x,v_x)=0$ and
$$\|v_x\|_{L^2([0,T];\RR)}\le K.
$$
 However, this
constant is given by continuity properties of the system
\eqref{eq-0-c-1}.
To be more precise,
since \eqref{eq-0-c-1} is exactly controllable the mapping
$\Phi_T: L^2(0,T; \RR) \rightarrow H$ defined by
$$\Phi_T(v)=\int_0^T e^{-(T-s)A}B v(s) ds $$ is  invertible. Furthermore, it is bounded thanks to our assumptions on the semigroup generated by $A$ and on the operator $B$.
Hence,
its inverse $\Phi_T^{-1}$ is also a bounded  operator. Let
$x\in H$ and $y\in H$, we have $$u^c(T,x, v_x
)=y=S(T)x+\Phi_T(v_x).$$ From this identity we infer that \begin{align*}
\int_0^T |v_x(s)|^2 ds\le \lVert \Phi_T^{-1}\rVert \cdot |u^c(T,x,
v_x)-S(T)x|\\
\le \lVert \Phi_T^{-1}\rVert |y-S(T)x|.\end{align*}
\end{proof}

\begin{proof}[{\bf Proof of Theorem \ref{nondeg2cad}}]
The proof starts by the same consideration as the proof of Theorem \ref{nondeg22}.
We list here only the points where the proof differs.

First, choose $R\ge r_1$ such that
\DEQSZ\label{RK-chosen}
\lk(\frac 3\delta\rk)^2 \,C\,\TT  R^{1-\frac 2\alpha}\, C_\sigma(1+K)\, \le \frac 12,
\EEQSZ
where $K:=\int_0^T \EE |u(s)|_H^2 \, ds$. By the assumptions on $B$ and $\sigma$,
$K$ is finite.
The next difference is that one has to find a predictable transformation
 $\theta:\Omega\times [0,\TT ]\rightarrow \CH $ such that
$$
{v(s)\over \sigma(u(s-))}  + {g_\ww } = \int_{\CH \setminus  B_\CH(\vardelta)} \lk[ c(z)
- c(\theta(s,z))\rk] \lambda (dz),\quad s\in[0,\TT ].
$$
Again, the existence of such a transformation is given by Lemma
\ref{trans_ex}. 

Let $\mu_\theta$ the following random measure defined by
$$
\mu_\theta:\CB(\CH)\times \CB([0,T])\ni(A\times I) \mapsto
  \int_{\RR^2}\int_I
1_A(\theta(s,z))\mu(dz,ds).
$$
Again, let $\QQ$ be the probability  measure on $\mathfrak{A}$ 
such that $\mu_\theta$ has compensator $\gamma $. Then, the
process $u_\mu^\theta $ defined by
\DEQSZ\label{eqn-0theta1} \lk\{
\barray d u^\theta_\mu(t,x) &=& A u^\theta_\mu(t,x) dt+
\int_\CH B c(z)\, \sigma( u(s-))\, ( \mu_\theta -\gamma ) (dz,dt) \\
u^\theta_\mu(0,x) &=& x. \earray \rk. \EEQSZ has under $\QQ$ the
same law as $u$, in particular
$$
\EE^\PP \lk[ 1_{[0,\delta]}(|u(\TT ,x_0)|)\rk] = \EE^\QQ \lk[
1_{[0,\delta]}(|u^\theta_\mu(\TT ,x_0)|)\rk] .
$$
Setting
$$v_\sigma(t):= {v(s)\over \sigma(u(s-))}
$$
again, due to Lemma \ref{density} the  density process
$\CG_\theta(t)={d\QQ^\theta_t\over d\bar \PP_t}$ satisfy the
following stochastic differential equation
\DEQS
 \lk\{\barray d\CG_\theta (t) &=& \CG_\theta(t)\lk(  \rho_z(  \kappa(|
v_\sigma(s) |),z ) ) \sgn( v_\sigma(s )) \rk)\,(\mu-\gamma )(dz,ds)
\\ \CG_\theta(0)&=& 1.
\earray\rk. \EEQS
Similarly, by Corollary \ref{theta-ex} we have
\DEQS {\EE^{\bar \PP}}    \sup_{0\le s\le r_1} |\CG_\theta(s)|
  & \le & 1 + {\EE^{\bar \PP}} \sup_{0\le s\le r_1} |\CG _\theta
(s)| \times
  \int_0 ^ {r_1}    | v_\sigma(s)|^2 \, ds.
   \EEQS
Now, if $ \int_0 ^ {r_1}    | v_\sigma(s)|^2 \, ds\le \frac 12$, we get
\DEQS {\EE^{\bar \PP}}    \sup_{0\le s\le r_1} |\CG_\theta(s)|
  & \le & 1\times (\frac 12)^{-1}.
   \EEQS
Since $v\in L^ 2([0,T];\RR)$ and $\sigma$ is bounded from below by  $C_\sigma$ we know
there exists a constant $C_v=  |v|_{ L^ 2([0,T];\RR)}/C_\sigma>0$ such that $\bar \PP$--a.s.\
$$ |v_\sigma|_{ L^ 2([0,T];\RR)}\le C_v.
$$
Arguing as before,
we  conclude that there exists a constant $C>0$
%
\DEQSZ\label{density11} {\EE^{\bar \PP}} \sup_{0\le s\le T}
|\CG_\theta(s) |&\le & C\, 2^{|v_\sigma|^2_{L^2([0,T];\RR)}}\sim  C\, 2^{C_v}   .
 \EEQSZ

On the other hand we know that under $\bar\PP$  the process
$u^\theta_\mu$ follows the following differential equation
\DEQSZ\label{eqn-0cc}
\\
\nonumber
\lk\{ \barray    
d u^\theta_\mu(t,x) 
&=&A u^\theta_\mu(t,x) dt+ \int_\CH B\sigma(u^\theta_\mu(t-,x)) \lk[ c(\theta(t,z))-c(z)\rk] (\mu-\gamma ) (dz,dt)  
\\&&{}\hspace{-2cm}
 +\int_\CH  B\sigma(u^\theta_\mu(t-,x))\lk[c(\theta(t,z))-c(z)\rk]\gamma(dz,dt)+
\int_\CH B \sigma(u^\theta_\mu(t-,x))c(z) ( \mu-\gamma ) (dz,dt),
\\ u^\theta_\mu(0,x) &=& x .
\earray\rk. \EEQSZ
Hence we can again write
\DEQS
 \bar \PP\lk( |u(\TT ,x)|\le \delta\rk)=\QQ( |u^\theta_\mu(\TT ,x)|\le \delta) = \EE^\QQ \lk[1_{|u_\mu^\theta (\TT ,x)|\le \delta}\rk]
= \EE^{\bar \PP}\lk[   \CG_\theta(\TT ) 
1_{|u_\mu^\theta (\TT ,x)|\le \delta}\rk] . \EEQS \del{\coma{The
processes $u$ and $u^\theta_\mu$ have the same law, therefore the
above equation should be replaced by } \red{\DEQS
  \PP\lk( |u(\TT ,x)|\le \delta\rk)=\QQ( |u^\theta_\mu(\TT ,x)|\le \delta) = \EE^\QQ \lk[1_{|u_\mu^\theta (\TT ,x)|\le \delta}\rk]
= \EE^{\bar \PP}\lk[   \CG_\theta(\TT ) 
1_{|u_\mu^\theta (\TT ,x)|\le \delta}\rk] . \EEQS}}
By the inverse H\"older inequality we get  
\DEQS
 \bar {\PP}\lk( |u(\TT ,x)|\le \delta\rk)
 &\ge &{ \EE^{{\bar \PP}}\lk[1_{|u_\mu^\theta (\TT ,x)|\le \delta}\rk]\over \EE^{{\bar \PP}}\lk[  \CG_\theta(\TT ) \rk]}.
\EEQS
Again, the denominator, i.e.\ $\EE^{\bar \PP}\lk[ |\CG_\theta(\TT )
|\rk]$, is bounded. In particular, we have by the assumption on $\sigma$,  \eqref{density11}
and Hypothesis \ref{hypo3} that
$$
\EE^{\bar \PP}\lk[  |\CG_\theta(\TT ) |\rk]
 \le {C'\over C_\sigma} \, \exp\lk(r_1^ { \alpha \beta_1+1\over \alpha} \int_0^T |v(s)|^2 ds\rk).
$$
The next steps are similar to the lines of the proof of Theorem \ref{nondeg22}.
To be more precise, we obtain first
\DEQSZ\nonumber
\lqq{ \EE^{\bar \PP}\lk[1_{|u_\mu^\theta (\TT ,x)|\le \delta}\rk] }&&\\
&\ge &{\EE^{\bar\PP}}\lk[1_{|\uc (\TT ,x,v)|\le \frac \delta
3}\,\,1_{|\uc (\TT ,x,v)- u_{\mu}^\theta (\TT ,x)   |\le \frac
\delta3} \rk] = \EE^{\bar \PP}\lk[1_{|\uc (\TT ,x,v)-
u_{\mu}^\theta (\TT ,x)   |\le \frac \delta3} \,\rk]
.\label{difference-2} \EEQSZ
Again, rewriting the difference $\Delta(\TT )=
u_{\mu}^\theta (\TT ,x)  -\uc (\TT ,x,v)$ as follows
\DEQS \lqq{ \Delta(\TT ) =
 \int_0^{\TT }  \int_\CH  e^{-(t-s)A}B\sigma(u(s-)) c(\theta(t,z))  (\mu-\gamma ) (dz,ds)}
 &&\\
&&{}
 + \int_0^{\TT }e^{-(t-s)A}B\sigma(u(s-)) g_R \, ds,
\EEQS
we have \DEQSZ\nonumber \EE^{\bar \PP}\lk[1_{|u_\mu^\theta (\TT
,x)|\le \delta}\rk] &\ge &
 \EE^{\bar \PP}\lk[1_{|\Delta(\TT ) 
    |\le \frac \delta3}
\,\rk]. \label{difference_22} \EEQSZ
Continuing as before, we get
%
\DEQS \lqq{ \bar \PP\lk( |\Delta(\TT )|\le \frac \delta3\rk)
\ge  \bar  \PP \lk( \mu(   \CR_\CH(\ww ) \times [0,\TT ])=0\rk)} &&
\\ &\times &\bar \PP
\Bigg(  \lk|\int_0^{T_\vardelta} \, \int_{ \CH\setminus
\CR_\CH(\ww )} e^{-(t-s)A}B \sigma(u(s-))c(\theta(s,z))(\mu-\gamma )(dz,ds)\rk|
\le {\delta\over 3 }\,\,\Big|\,\, \mu(  \CR_\CH (\ww ) \times
[0,\TT ])=0\Bigg)
\\ &\ge &  \bar  \PP \lk( \mu(   \CR_\CH(\ww ) \times [0,\TT ])=0\rk)\times  \Bigg[ 1-
\\ &&\bar \PP
\Bigg(  \lk|\int_0^{T_\vardelta} \, \int_{ \CH\setminus
\CR_\CH(\ww )} e^{-(t-s)A}B \sigma(u(s-))c(\theta(s,z))(\mu-\gamma )(dz,ds)\rk|
> {\delta\over 3 }\,\,\Big|\,\, \mu(  \CR_\CH (\ww ) \times
[0,\TT ])=0\Bigg) \Bigg]
\\
 &:= &  \bar  \PP \lk( \mu(   \CR_\CH(\ww ) \times [0,\TT ])=0\rk)\times  \lk[ 1- I(T)\rk].
 \EEQS
If we can show, that $I(T)\le \frac 12$, we are done.
Again, we apply the  Chebyscheff inequality to estimate $I(T)$ and use the fact that the random variables $\mu(
\CR_\CH(\ww ) \times [0,\TT ])$ and $\mu( A \times [0,\TT ])$ for
all $A\in \CB(\CH\setminus \CR_\CH(\ww ))$ are independent. Doing so we obtain
\DEQS I(T)
 &\le &\frac 3 \delta \times
\EE ^{\bar \PP}   \lk| \int_0^{T_\vardelta} \, \int_{ \CH\setminus
\CR_\CH(\ww )} e^{-(t-s)A}B\sigma(u(s-)) c(\theta(s,z))(\mu-\gamma )(dz,ds)
\,\rk| ^2 . \EEQS
The Burkholder inequality, the fact that $c(\theta(t,z))\le c(z)$ and the linear grow condition on $\sigma$ give 
\DEQS
 I(T)
 &\le &\frac 3 \delta \times
 \EE^{\bar \PP}\lk[  \int_0^{\TT } \, \int_{ \CH\setminus
\CR_\CH(\ww )} \lk| e^{-(t-s)A}B\sigma(u(s-)) c(z)\rk|^ 2 \lambda \, ds\rk]
\\
&\le &C\, \frac 3 \delta \times
\TT  \ww  ^{1-\frac 2\alpha}C_\sigma \lk( 1+ \EE ^{\bar \PP} \int_0^{\TT } \lk| u(s)\rk|^2 \, ds \rk) .
 \EEQS
 Therefore, collecting all together
  \DEQS \EE^{\bar \PP}\lk[1_{|\uc (\TT
,x,v)- u_{\mu}^\theta (\TT ,x)   |\le \frac \delta3}\rk]&\ge&
\frac{1}{C(R)}\lk(1-\lk(\frac 3\delta\rk)^2 \,C\,\TT  \ww  ^{1-\frac
2\alpha} C_\sigma (1+K) \rk). \EEQS
Since $R$ satisfies \eqref{RK-chosen}
we get
\DEQS 
 \EE^{\bar \PP}\lk[1_{|u_\mu^\theta (\TT ,x)|\le \delta}\rk]
&\ge &
 \frac 12 \,\, e^{-\lambda (\CR_\CH(\ww ))\TT  } C(R).
 \EEQS
Hence, 
 we have shown that
\DEQSZ\label{ending2}
 \bar {\PP}\lk( |u(\TT ,x)|\le \delta\rk) \ge { \frac 12 \,\, e^{-\lambda (\CR_\CH(\ww ))\TT  } } C(R) \exp(-|v|_{L^2 ([0,T];\RR)}^2 ),
\EEQSZ
which gives the assertion.
\end{proof} 

\begin{proof}[{\bf Proof of Theorem \ref{nondeg2cad2}}]
%
%
Let $x\in\CB_H(C)$ and  $u^c$ the solution to
\DEQSZ\label{eq-0-c-11} \lk\{ \barray \dot u^c(t,x,v) &=& A \uc
(t,x,v_x)+ B v_x(t),\quad t\ge 0,
\\
\uc (0,x,v_x)&=& x, \earray\rk. \EEQSZ
In order to show Theorem \ref{nondeg2cad2} we use the same arguments as we have used in the proof of Theorem \ref{nondeg22}.
That means,
first, $R$ can be chosen in such a way, that
  \eqref{RK-chosen}  holds for all initial conditions $x$. Since $K$ depends linearly on $x$, this will not be a problem.
Secondly,
 we have to show that the RHS of Inequality \eqref{ending2}
can be estimated from below for all $x\in\CB_H(\delta)$.
That means, we have to show that
there exists a constant $C'>0$ such that for any $x\in\CR
_H(C)$
there exists a control $v_x\in L^2([0,T];\RR)$ such that
$u^c(T,x,v_x)=0$ and
$$\|v_x\|_{L^2([0,T];\RR)}\le C'.
$$
 However, again, this
constant is given by continuity properties of the system
\eqref{eq-0-c-11}.
\end{proof}

\section{Uniqueness of the invariant measure and the asymptotic
strong Feller property}\label{sec:3}

Let $u$ be the solution of the following stochastic evolution equation 
\DEQSZ\label{eqn-00-inv}
\lk\{ \barray du(t,x) &=& Au(t,x) + \int_\CH B \sigma( u(t,x))\, z
\tilde \eta(dz,ds), \\
u(0,x) &=& x\in H.
\earray\rk.
\EEQSZ
Let us assume in addition that $\sigma$ is bounded. In particular, there exists a $K_\sigma$ such that $|\sigma(x)|\le K_\sigma$, for all $x\in H$.
In case, $u$ is not \cadlag \ in $H$, $\sigma$ is supposed to be a constant.

If the semigroup generated by $A$ is of contractive type, i.e.\
there exists a $\omega>0$ and $M>0$ such that $\|
e^{-tA}\|_{L(H,H)}\le M e^{-\omega t}$, $t\ge 0$, then by direct
calculations the following a'priori estimate can be shown
\DEQSZ
\label{aprioribound} \EE \lk|u(t,x)\rk|^2_H \le e^{-\omega t} M\EE
|x|^2_H + M\, C(t,\gamma) \int_\CH |z|^2 \,\nu(dz)
 \EEQSZ
 where
$C(t,\gamma) =K_\sigma^2 \int_0^t e^{-\omega s} s^{-2\gamma}\, ds$. Note that
$\lim_{t\to \infty}  C(t,\gamma)<\infty$.

\del{Now, the existence of
the invariant measure can be shown by an application of the
Krylov-Bogoliubov Theorem.}

Here, the existence of the invariant measure can be shown by an
application of the Krylov-Bogoliubov Theorem (see \cite[Theorem 3.1.1]{DZ97}
). First, we will define for $T>0$ and $x\in \CCH$
the following probability measure on $(H,\CB(H))$
\DEQSZ \CB(\CCH ) \ni \Gamma \mapsto R_T(x,\Gamma) := \frac 1T
\int_0^T \, P_t(x,\Gamma)\, dt.
\EEQSZ In addition, for any $\rho\in M_1(\CCH )$, let
$R^\ast_T\rho $ be defined by
$$
\CB(\CCH ) \ni \Gamma \mapsto \int_\CCH
R_T(x,\Gamma)\,\rho(dx).
$$
Corollary 3.1.2 in \cite{DZ97} says, that if for some probability
measure $\rho$ on $(\CCH ,\CB(\CCH ))$ and some sequence
$T_n\uparrow \infty$, the sequence
$\{R^\ast_{T_n}\rho:n\in\NN\}$ is tight, then there exists an
invariant measure for $(\CP_t)_{t\ge 0}$.
%
That means, it is sufficient for the existence of an invariant
measure to show that for all $\ep>0$ for all $x\in\CCH $, there
exists a compactly embedded subspace $E\hookrightarrow H$, a $R>0$
such that we have for all $T>0$ \DEQSZ\label{oben}
  \frac 1T \int_0^T \,\PP(|u(t,x)|_E\ge R )\, dt<\ep.
\EEQSZ
  However, if there exists a constant $C>0$ and a number $p>0$ such that
  $$
\sup_{0\le t\le T }  \EE  |u(t,x)|^p_E\le C,\quad T\ge 0,
  $$
  then \eqref{oben} holds. 

Observe, if $A$ generates a strongly continuous semigroup
$(e^{-tA})_{t\ge 0}$ of contraction on $\CCH$,
then  
$$
\int_0^\infty \lk| e^{-tA} \rk|^2_{L(H,D(A^{-\gamma}))}  \,\,
dt<\infty
.
$$

\begin{theorem}\label{exinv}
Let $H_1$ be a Hilbert space such that $H_1\hookrightarrow H$
compactly, Assume $A$ generates a strongly continuous semigroup
$(e^{-tA})_{t\ge 0}$ of contraction on $\CCH$ and $H_1$,
 and
$$\int_0^ \infty |S(t-s)B|_{L(D(A^ {-\gamma},H_1)}^2\,ds<\infty, $$
Then, if
$$
\int_\CH |z|^2\nu(dz)<\infty, $$ and $\sigma$ is bounded,
then the Markovian semigroup
 $(\CP_t)_{t\ge
0}$ admits an invariant measure.
\end{theorem}

\begin{proof}
First, Equation \eqref{{eqn-00-inv}} can be written as follows \DEQS
\lk\{\barray du(t,x)& = &Au(t,x)\: dt + \int_\CZ B \, \sigma(u(s-))\, z\,\tilde
\eta(dz,ds) ,\quad t>0,
\\ u(0,x)& =&x.\earray\rk.
\EEQS
If $A$ is a semigroup of contractions then 
 (see \cite{levy2})
$$
\int_0^\infty \int_\CZ \lk| e^{-tA} Bz\rk|^2\nu(dz) \, dt<\infty
$$
Now,  we will show that there exists a constant $C>0$ such that
 $$ \sup_{0\le t\le T }  \EE |u(t,x)|^2_{H_1}\le
C,\quad T\ge 0.
  $$
Due to standard arguments (see \cite{levy2}) we get
 \DEQS \EE \lk| u(t,x)\rk|^2 _{H_1}&\le& C_1 \lk|e^{-tA}x\rk| _{H_1}^ 2 + C_2 \EE \lk|
\int_0^t e^{-(t-s)A}\, B\, \sigma(u(s-))\, z \nu(dz) ds\rk| _{H_1}^2
\\ &\le& C_1\lk|
e^{-tA}x_0\rk| _{H_1} ^ 2 +{}  C_2K_\sigma \EE \int_0^t
\int_\CH\lk|e^{-(t-s)A}\, Bz\rk| _{H_1}^2 \nu(dz)\, ds . \EEQS
Due  to the assumption, the first and the second summand are
bounded uniformly for all $t\ge 0$. Hence,  there exists a
constant $C>0$ such that
\DEQS\sup_{0\le t<\infty}
 \EE \lk|u(t,x)\rk|^2 _{H_1}  &\le& C.
\EEQS Now, the Chebyshev inequality leads for any $R>0$ to
\DEQS
  \frac 1T \int_0^T \,\PP(|u(t,x)| _{H_1} \ge R )\, dt & \le &  \frac 1T \int_0^T \, {\EE \lk| u(t,x)\rk| _{H_1} ^2\over R^2}\, ds
 \\\frac 1T \int_0^T \, {\EE \lk| u(t,x)\rk| _{H_1}  ^2\over R^2}\, ds &\le&
 {C\over R^2}
   .
\EEQS Given $\ep>0$ and taking $R>(\frac  C\ep)^{\frac 12}$,
inequality \eqref{oben} follows.
\end{proof}

\del{
 (see \cite{DZ97}). For details on the
Markovian semigroup and existence of the invariant measure, we
refer to the appendix \ref{existence-inv}.}

By \cite[Corollary 3.17]{mat-hai-2006} we know that if the
semigroup is asymptotically strong Feller (for the definition  of
asymptotic strong Feller, we refer to \cite{mat-hai-2006}) and
there exists a point $x\in H$ such that $x\in supp(\sigma)$,
whenever $\sigma$ is an invariant measure of the Markovian
semigroup $(\CP_t)_{t\ge 0}$, then the semigroup $(\CP_t)_{t\ge
0}$ admits at most one invariant measure.
%
%

To be more precise, assume for the time being that the semigroup
$(\CP_t)_{t\ge 0}$ is asymptotically strong Feller. Hence, it
remains to prove that there exists a $x\in H$ such that   $x\in
supp(\sigma)$. Now, the two following properties imply that $0\in
supp(\sigma)\footnote{Note, that $x\in supp(\sigma)$ iff for all
$\delta >0$, $\sigma( \CR_H(\delta))>0$.}$  whenever $\sigma$ is
an invariant measure.
 \begin{itemize}
   \item There exists a constant $C>0$ such that
\DEQSZ\label{inv-1} \inf_{\{\rho \mbox{ \small  is an invariant
measure}\}} \rho( \CR_H( C)) > 0. \EEQSZ
   \item For all $\delta>0$ and for all $x\in \CR_H(C)$ there exists a time $T_\delta>0$ and some $\kappa>0$  such that
\DEQSZ\label{inv-2} {\PP}\lk( u(T_\delta,x)\in\CR_H(\delta)\rk)\ge
\kappa. \EEQSZ
 \end{itemize}
It follows that $0\in supp(\sigma)$ by the following observations.
Since $\sigma$ is invariant we have
$$
\sigma\lk( \CR_H(\delta)\rk) \ge  \kappa \cdot \inf_{\{\rho \mbox{
\small is invariant measure}\}} \rho(\bar \CR_H( C)).
$$
Now, estimates \eqref{inv-1} and \eqref{inv-2} give the assertion.

Inequality \eqref{inv-2} can be verified by Theorem \ref{nondeg2}.
Estimate \eqref{inv-1} follows by the fact that for any invariant
measure
 $\rho$  of the semigroup $(\CP_t)_{t\ge 0}$
there exists a constant $C>0$ such that \DEQSZ\label{bound} \int_H
|u|^2_H\, \rho(du) \le C. \EEQSZ Estimate \eqref{bound}  follows
by the same lines as in \cite[Lemma B.1]{EMS} and the a priori
estimate \eqref{aprioribound}. Now, an application of the
Chebyscheff inequality leads to \eqref{inv-1}.

It remains to investigate under which conditions the semigroup is
asymptotical strong Feller. However, before continuing we
introduce a second concept of controllability. Again,  $H$ denotes
a Hilbert space, $A:H\to H$ a generator of a strongly continuous
semigroup $(e^{-tA})_{t\ge 0}$ on $H$ and $B:\CH \to H$. Then we
say that the system \DEQS \lk\{ \barray \dot u^c(t,x,v) &=& A \uc
(t,x,v)+ B v(t),\quad t\ge 0,
\\
\uc (0,x,v)&=& x, \earray\rk. \EEQS
 is {\em null controllable with vanishing energy}
(see \cite{PPZ,priola-3}), if it is null controllable and for any
$x\in H$ there exists a sequence of times $\{t_n\ge 0:n\in\NN\}$
and a sequence of controls $\{ v_n:n\in\NN\}\subset
L^2([0,t_n];\CH )$ such that $u(t_n,x,v_n)=0$ for any $n\in\NN$
and
$$
\lim_{n\to \infty} \int_0^{t_n} |v_n(s)|_\CH ^2 \, ds =0.
$$

\del{In Theorem \ref{theo-asf} we verify some conditions under
which the semigroup is asymptotically strong Feller, and, in
Theorem \ref{nondeg2} we show under which conditions the semigroup
is non degenerate. From these two theorems follows the uniqueness
of the invariant measure. In section \ref{dampedwave} we apply the
results to the damped wave equation described in Example
\ref{damped-ex}. }


\medskip

Now we are ready to present our result about the asymptotically
strong Feller property.
\begin{theorem}\label{theo-asf}
If $\alpha>1$,
and if the system  \eqref{eq-0-c} is null controllable with vanishing energy, then the Markovian semigroup of system
\eqref{eqn-00-inv} is asymptotically strong Feller.
\end{theorem}
\begin{coro}\label{corouniqe}
Assume that $\alpha>1$ and that the conditions of Theorem
\ref{exinv} and Theorem \ref{theo-asf} are satisfied.
\del{the system  \eqref{eq-0-c} is null controllable with vanishing
energy.}
Then, the Markoivan semigroup $(\CP_t)_{t\ge 0}$  generated by the solution to 
\eqref{eqn-00-inv}  admits an invariant measure, and
this invariant measure 
is unique.
\end{coro}

\begin{proof}[{\bf Proof of Theorem \ref{theo-asf}}:]\label{pot-theo-asf}
%
Again, we will switch for technical reasons to another
representation of the Poisson random measure.
 Let
$\overline{\mathfrak{A}}=(\bar\Omega,\bar\CF,\bar\FF,\bar\PP)$ be
a filtered probability space with filtration
$\bar\FF=(\bar\CF_t)_{t\ge 0}$ and let $\mu$ be a Poisson random
measure on $\CH$ over $\overline{\mathfrak{A}}$ having intensity
measure $\lambda$ (Lebesgue measure).
The compensator of $\mu$ is denoted by $\gamma $ and given by
$$\CB(\CH)\times\CB([0,\infty))\ni A\times I\mapsto \gamma ( A\times I) := \lambda (A)\,\lambda(I).
$$
%
%
Let \DEQSZ\label{trans-cc1} c:\RR ^+ \ni r \mapsto  \sup_{\rho>0}
\lk\{\int_\rho^\infty k(s)\, ds \ge r \rk\} \, &\mbox{ if } r>0.
\EEQSZ
Now, the stochastic evolution equation given in  \eqref{eqn-00}
reads as follows \DEQSZ\label{eqn-00-mc1} \lk\{ \barray du(t,x)
&=& Au(t,x) + \int_\CH B  \sigma(u(s-))\, c(z)
 \tilde \mu(dz,ds), \\
u(0,x) &=& x\in H.
\earray\rk.
\EEQSZ
We split the proof in several steps.

\paragraph{\bf Step I:}
Fix $x\in H$. In order to show the asymptotical strong Feller
property of $(\CP_t)_{t\ge 0}$ in $x$, we have to show that there
exist an increasing sequence $\{t_n:n\in\NN\}$ and a  totally
separating sequence of pseudometrics $\{d_n:n\in\NN\}\footnote{An
increasing sequence $\{ d_n:n\in\NN\}$ of pseudo metrics is called
a {\em totally separating system} of pseudo metrics for $\CX$, if
$\lim_{n\to\infty}d_n(z,y) =1$ for all $z,y\in\CX$, $z\not = y$.
}$ such that\footnote{Let $d$ be a pseudo--metric on $\CX$, we
denote by $L(\CX,d)$ the space of $d$-Lipschitz functions from
$\CX$ into $\RR$. That is, the function
$\phi:\CX\to\RR$ is  an element of $L(\CX,d)$ if  
$$
\|\phi\|_d := \sup_{z,y\in\CX\atop z\not = y } {
|\phi(z)-\phi(y)|\over d(z,y)}<\infty.
$$
}\footnote{For a pseudo-metric $d$ on $\CX$ we define the
distance between two probability measures $\CP_1$ and $\CP_2$ wrt
to $d$ by
$$
\| \CP_1-\CP_2\|_{d} := \sup_{\phi\in L(\CX,d)\atop\|\phi\|_d=1}
\int_\CX \phi(x)\, (\CP_1-\CP_2)(dx).
$$}
\DEQSZ\label{zuzeigen} \lim_{\ep>0} \limsup_{n\to\infty}
\sup_{y\in B(x,\ep)} \lk\|
\CP_{t_n}(x,\cdot)-\CP_{t_n}(y,\cdot)\rk\|_{d_n}=0. \EEQSZ
Let $\{a_n:n\in\NN\}$ be a sequence of positive real numbers
converging to zero. 
Let
$$
d_n(y,z) := 1\wedge ( |z-y|_H/a_n ),\quad z,y\in H,\quad n\in\NN.
$$
\del{\begin{remark}
The sequence $\{d_n:n\in\NN\}$ is
a  totally separating
sequence of pseudometrics.
\end{remark}}
%
%
Fix $h\in H$ with $|h|=1$. Since the system \eqref{eq-0-c} is null controllable with vanishing energy we can find a sequence of times $\{ t_n:n\in\NN\}$ and controls
$\{v^n:n\in\NN\} $ such that $a^2_n\ge \int_0^{t_n}|v^n (s)|^2
\, ds$, $n\in\NN$, and the solution $u^c$ to
\DEQS
\lk\{ \barray du^c(t,x,v^n) &=& A u^c(t,x,v^n)\, dt + B v^n(t) \, dt ,\quad 0\le t\le t_n
\\
u^c(0,x,v^n) &=& x,\earray\rk.
 \EEQS
satisfy $u^c(t_n,x,v^n)=u^c(t_n,x+ h,0)$. \del{\com{The last
condition is equivalent to say that $$e^-{t_n
A}h=\int_0^{t_n}e^-{(t_n-s) A}Bv_n(s) ds. $$}}
For simplicity, put $y=x+\ep h$ and $v^n_\ep:=\ep\cdot v^n$. Then,
it follows by the linearity that $u^c(t_n,x,\ep v^n)=u^c(t_n,x+
\ep h,0)=u^c(t_n,y,0)$. In order to give an estimate of
 \DEQSZ
\label{diffquotient} \lk\|
\CP_{t_n}(x,\cdot)-\CP_{t_n}(y,\cdot)\rk\|_{d_n} \del{ = \lk| \EE
^ \PP\lk[\phi_n\lk( u(t_n, x+\ep h) \rk)\rk] -\EE ^ \PP\lk[
\phi_n\lk( u(t_n,x)\rk)\rk]\rk|}
 \EEQSZ
 in terms of $\ep$ and $n$, we define the following
sequence of continuous functions. Let $\phi\in C_b(H)$, there
exists a sequence $\{\phi_n:n\in\NN\}$, $\phi_n\in C^\infty_b(H)$,
such that $\phi_{n}\to\phi$ pointwise, $\|\phi_n\|_\infty\le
\|\phi\|_\infty$ and $\|\phi_n\|_{d_n}\le 1$ for all $n\in\NN$.
Now, the following identity holds 
\DEQS \lk\|
\CP_{t_n}(x,\cdot)-\CP_{t_n}(y,\cdot)\rk\|_{d_n}
{ = \lk| \EE ^ \PP\lk[\phi_n\lk( u(t_n, x+\ep
h) \rk)\rk] -\EE ^ \PP\lk[ \phi_n\lk( u(t_n,x)\rk)\rk]\rk|}.
\EEQS
Hence, we have to show
\DEQSZ\label{zuzeigen-1}
\lim_{\ep>0} \limsup_{n\to\infty} \sup_{y\in B(x,\ep)} \lk| \EE ^ \PP\lk[\phi_n\lk( u(t_n, x+\ep
h) \rk)\rk] -\EE ^ \PP\lk[ \phi_n\lk( u(t_n,x)\rk)\rk]\rk|=0.
\EEQSZ
\paragraph{\bf Step II:}
Next, let us introduce a transformation
$\theta_n^\ep:[0,\infty)\times \CH \to\CH $, such that we have
%
%
$$  \int_\CH
(c(\theta^n _\ep(s,z))-c(z))\lambda (dz)=\frac{v^n_\ep(s)}{\sigma(u(s-)) } \quad \mbox{for
all}\quad s\in[0,t_n].
$$
From now on we denote $$ v^n_{\ep,\sigma}(s)=\frac{v^n_\ep(s)}{\sigma(u(s-))}.$$
%
%
%
%
In fact, by Lemma \ref{trans_ex} we can suppose that such a
transformation $\theta^n _\ep$ exists and is given by (see definition \eqref{fidef})
\DEQSZ \label{theta-def} \theta^n_\ep:[0,\infty)\times \RR\ni
(s,z)\mapsto   z+ \hh(\kappa (|v_\epsi^n(s)|),z)\sgn( v_\epsi^n(s)) , \EEQSZ
Here $\kappa$ denotes the inverse of $\wt$ and is defined on page \pageref{def-kappa}.
In addition, let   $\mu^n _\ep$ be a random measure defined
by 
\DEQSZ \CB(\CZ)\times \CB([0,t])\ni A\times I\mapsto \mu^n
_\ep(A\times I) = \int_I\int_A 1_A(
{\theta_\ep^n}(s,z))\mu(dz,ds)
 \EEQSZ
Let $\QQ^{\ep,n}$ be that probability measure on $\overline{\FA}$
such that $\mu^n _\ep$ has compensator {$\gamma=\lambda\cdot
\lambda $\footnote{$\lambda$ denotes the Lebesgue measure.}}. Let $u_\ep^n $ be the solution to
\DEQS \lk\{ \barray  du_\ep^n (t,x) &=& A u_\ep^n (t,x)\, dt +
\int_\CH B \sigma(u(t-))[c(\theta^n _\ep(t,z))-c(z)  ]  \mu(dz,dt)
\\
&&\quad {} + \int_\CH B\sigma(u(t-))
c(z)( \mu-\gamma )(dz,dt),
\\
u^n _\ep(0,x) &=& x, \earray\rk.
\EEQS
and let $u_{\mu,n,\ep}^c$ be solution to
\DEQS \lk\{ \barray du_{\mu,n,\ep}^c(t,x,v_\epsi^n) &=& A
u_{\mu,n,\ep}^c (t,x,v^n _\ep)\, dt + B v^n _\ep(t)\, dt +\int_\CH
B \sigma(u(t-)) c(z) (
 \mu-\gamma )(dz,dt),
\\
u_{\mu,n,\ep} ^c (0,x,v^n_\ep) &=& x. \earray\rk. \EEQS
Observe that, first, by the choice of the transformation $\theta_n^\ep$
under $\QQ^{\ep,n} $ the random variable  $u_\ep^n (t_n,x)$ is identical in law to
the process $u(t_n,x)$. In particular, we have
$$
\EE^{\QQ_{t_n} ^{\ep,n}} \lk[
\phi_n \lk( u_\ep^n ({t_n},x) \rk)\rk]= \EE ^ {\bar \PP}\lk[ \phi_n(u(t_n,x))\rk].
$$
Secondly, by the choice of the control
and the linearity of $A$ we have $u_{\mu,n,\ep}^c(t_n,x,v^n_\ep) = u(t_n,x+\ep
h)$.
For $t\ge 0$ let $\QQ^{\ep,n} _t$ be the restriction of
$\QQ^{\ep,n} $ onto $\bar\CF_t$ and $\bar \PP_t$ be the
restriction of $\bar \PP$ onto $\bar \CF_t$.
Now we are ready to give an estimate of
$$\lk\|
\CP_{t_n}(x,\cdot)-\CP_{t_n}(y,\cdot)\rk\|_{d_n}.
$$
First, we  write
\DEQSZ \label{diffquotient1} \lk\|
\CP_{t_n}(x,\cdot)-\CP_{t_n}(y,\cdot)\rk\|_{d_n}
 = \lk| \EE ^ \PP\lk[\phi_n\lk( u(t_n, x+\ep
h) \rk)\rk] -\EE ^ \PP\lk[ \phi_n\lk( u(t_n,x)\rk)\rk]\rk|.
\EEQSZ
and
\DEQS \lqq{ \EE ^ {\bar \PP}\lk[\phi_n\lk( u({t_n}, x+\ep h)
\rk)\rk] -\EE ^
{\bar \PP}\lk[ \phi_n \lk( u({t_n},x)\rk)\rk] } &&\\
&=& \EE ^ {\bar \PP}\lk[\phi_n \lk(
u_{\mu,n,\ep}^c({t_n},x,v^n_\ep) \rk)\rk] - \EE ^ {\bar
\PP}\lk[\phi_n \lk( u^n _\ep({t_n},x) \rk)\rk] + \EE ^ {\bar
\PP}\lk[\phi_n \lk( u^n _\ep({t_n},x) \rk)\rk] -\EE^{\bar \PP}\lk[
\phi_n \lk( u({t_n},x)\rk)\rk]
\\
&=& \EE ^ {\bar \PP}\lk[\phi_n \lk( u_{\mu,n,\ep}^c({t_n},x,v^n
_\ep) \rk) - \phi_n \lk( u_\ep^n ({t_n},x) \rk)\rk] + \EE^{\bar
\PP}\lk[ \lk[ 1-{d\QQ_{t_n} ^{\ep,n}\over d{\bar \PP}_{t_n}}\rk]
\phi_n \lk( u_\ep^n ({t_n},x) \rk)\rk] \\ && \qquad {} +\EE^{\QQ^
{\ep,n}} \lk[\phi_n \lk( u_\ep^n ({t_n},x) \rk)\rk] -\EE^{\bar
\PP}\lk[ \phi_n \lk( u({t_n},x)\rk)\rk]
\\
&=& \EE ^ {\bar \PP}\lk[\phi_n \lk(
u_{\mu,n,\ep}^c({t_n},x,v^n_\ep) \rk) - \phi_n \lk(
u_\ep^n({t_n},x) \rk)\rk] + \EE^{\bar \PP}\lk[ \lk[ 1-{d\QQ_{t_n}
^{\ep,n}\over d{\bar \PP}_{t_n}}\rk] \phi_n \lk( u_\ep^n({t_n},x)
\rk)\rk] . \EEQS
Next, \DEQS \lqq{ \lk|\EE ^ {\bar \PP}\lk[\phi_n \lk( u({t_n},
x+\ep h) \rk)\rk] -\EE ^
{\bar \PP}\lk[ \phi_n \lk( u({t_n},x)\rk)\rk]\rk| } &&\\
&\le & {1 
\over a_n} \,\, \EE ^ {\bar \PP}\lk|
u_{\mu,n,\ep}^c({t_n},x,v^n _\ep)  - u_\ep^n ({t_n},x) \rk| +
||\phi_n ||_\infty\,\EE^{\bar \PP} \lk| 1-{d\QQ_{t_n} ^{\ep,n}\over
d{\bar \PP}_{t_n}}\rk|.
\EEQS
Let us put
$$
I^\ep_1 := \EE ^ {\bar \PP}\lk| u_{\mu,n,\ep}^c({t_n},x,v^n _\ep)  -
u_\ep^n ({t_n},x) \rk| , \quad \mbox{ and } \quad I^\ep_2 := \EE^{\bar
\PP} \lk| 1-{d\QQ_{t_n} ^{\ep,n}\over d{\bar \PP}_{t_n}}\rk|.
$$
Next, by the construction of $ u_\ep^n (t,x)$ and $u_{\mu,n,\ep}^c(t,x,v^n _\ep)$ we
see that
\DEQS u_\ep^n ({t_n},x)-u_{\mu,n,\ep}^c({t_n},x,v^n _\ep)= \int_0^{t_n} \int_\CH e^{-({t_n}-s)A} B \sigma(u(s-))\lk[ c(z)-c(\theta_n^\ep(s,z))\rk]\,(\gamma-\mu)(dz,ds). \EEQS
and therefore
\DEQS
\lqq{\EE^{\bar \PP }\lk|u^n _\ep ({t_n},x)-u_{\mu,n,\ep}^c({t_n},x,v^n _\ep)\rk|^2 }
& & \\
&\le & \EE^{\bar \PP } \int_0^{t_n} \int_\CH \lk|e^{-({t_n}-s)A}
B\sigma(u(s-))\lk[ c(z)-c(\theta_n^\ep (s,z))\rk] \rk| \lambda  (dz)\, ds
\\
&\le & \frac 1 C_\sigma \EE^{\bar \PP } \int_0^{t_n} \lk|e^{-({t_n}-s)A} B\rk| \left|v^n_\ep(s)\right| \,ds
. \EEQS
Hence,
\DEQS
 \lk| I _1^\ep\rk|&\le& \,  \EE ^ {\bar \PP} \lk| u({t_n},x) - u_{\mu,n,\ep}^c({t_n},x,v^n _\ep)
 \rk|^2
\le  
 \frac{C}{C_\sigma}\biggl(\int_0^{t_n} \lk|e^{-({t_n}-s)A} B\rk|^2 ds\biggr)^\frac 12  \EE^{\bar \PP }\biggl(\int_0^{t_n}\lk|v^n _\ep(s)\rk|^2 \,ds\biggr)^\frac12 
\\
&\le & \frac{C || B||^2}{C_\sigma}\lk( \int_0^{t_n} (t_n-s)^{-2\gamma}e^{-2({t_n}-s)\rho}\, ds\rk) ^\frac 12
\lk( \int_0^{t_n} \lk| v^n _\ep(s)\rk|^2 \,ds \rk)^\frac 12
%
\\ &\le & \frac{C(\gamma,\rho,B)}{C_\sigma}\,  \lk( \int_0^{t_n} \lk| v^n _\ep(s)\rk|_\CH ^2
\,ds\rk)^\frac 12 . \EEQS
%
To give an estimate of the second term $I ^ \ep_2$ we apply
 \cite[Theorem 1]{abso} to get an exact representation of the Radon
Nikodym derivative. In particular, we obtain
\DEQS
 I_2 ^ \ep &\le &\EE ^ {\bar \PP}\lk[(1-\tfrac{d\QQ_{t_n} ^{\ep,n}}{d{\bar \PP}_{t_n}})\, \phi\lk( u^n _\ep({t_n},x) \rk)\rk]
\\ &\le & \EE ^ {\bar \PP}\lk[(1-\CG^n _\ep({t_n}))\, \phi\lk( u_\ep^n ({t_n},x)\rk)\rk]
\EEQS
where $\CG^n _\ep $ is defined by (see Lemma \ref{density} and \eqref{theta-def}) 
\DEQSZ \label{asdonebefroe}
\\\nonumber
 \lk\{\barray d\CG^n _\ep (t) &=& \CG^n
_\ep (t)\lk( \rho_z( \kappa(|v_{\ep,\sigma}^n(s)|) z ) \sgn( v_{\ep,\sigma}^n(s)) \rk)
\,(\mu-\gamma)(dz,ds)
\\ \CG^n _\ep (0)&=& 1.
\earray\rk. \EEQSZ
 Now, the H\"older inequality gives
\DEQS I_2 ^ \ep &\le & \EE ^ {\bar \PP}\lk|1-\CG^n _\ep ({t_n})
\rk|\, |\phi|_\infty.
\EEQS
First we will give an estimate of $\EE    \sup_{0\le s\le t_n} |\CG
_{\ep}^n (s)|$. An application of the It\^o formula and the
estimate \eqref{integration-gives} give for $0\le t\le t_n$
\DEQSZ\label{her} \nonumber {\EE^{\bar \PP}}  \sup_{0\le s\le t}
|\CG^n _{\ep}(s)|
  & \le &1 + \int_0 ^ t   \int_{\RR^+} \,|\CG ^n _{\ep}
(s-)|  | \kappa ( v_{\ep,\sigma}^n (s) )|  |\rho_z(  | \kappa ( v_{\ep,\sigma}^n (s) )|   ,z 
)| \, ds \\
  & \le &1 + C\, \int_0 ^ t    \,|\CG ^n _{\ep}
(s-)|   \kappa( |v_{\ep,\sigma}^n(s)  |)   \int_{\RR^+}|\rho_z( | \kappa ( v_{\ep,\sigma}^n(s) )| ,z 
)| \,dz  ds  . \EEQSZ
By Corollary \ref{theta-ex} and assumption on $\sigma(\cdot)$
 it follows 
\del{\DEQS \int_\RR|\rho_z(k,z)|\, dz &=& \bcase \phantom{\Bigg|} k\,
\int_{r_1 k^ \neu}^ {2 r_1 k^ \neu} z^ {-\beta_1-1}\, dz\le C r_1^
{-\beta_2} k^{1-\beta_1 \neu } \mbox{ for } K\ge 1,
\\ \phantom{\Bigg|}k^ \frac 12 \,\int_{r_1 k^ \nb}^ {2 r_1 k^ \nb} z^
{-\beta_2-1}\, dz\le C r_1^ {-\beta_2}  k^{-\beta_2 \nb } \mbox{
for } K< 1 . \ecase \EEQS
By the choice of $\beta_i$ and $\gamma_i$, $i=1,2$ we have}
\DEQS\label{endlich}
   \int_{\RR^+}|\hh_z( \kappa( |v_{\ep,\sigma}^n(s)  |),z 
)| \,dz  \le \frac{C(r_1)}{C_\sigma} |v^n _\ep(s)  | ^ {2} . \EEQS
 Substituting this last estimate in \eqref{her} we obtain
\DEQSZ
{\EE^{\bar \PP}}  \sup_{0\le s\le t} |\CG^n _{\ep}(s)|
  & \le &1 +  \frac{C(r_1)}{C_\sigma}\int_0 ^ t   \,|\CG ^n _{\ep}
(s-)|  | v_{\ep}^n(s) | ^ 2 \, ds \\
  & \le &1 + \frac{C(r_1)}{C_\sigma}\,\EE \sup_{0\le s\le t}   |\CG^n _{\ep}(s)|
  \,  \int_0 ^ t   |v^n _\ep(s)  |^ 2
  ds  . \EEQSZ
Since $\int_0 ^ {t _n}  |v^n _\ep(s)  |^ 2
  ds \le a_n^2$ and $a_n\to0$, there exists a $n_0\in\NN$ such
  that $C(r_1) a_n^ 2<1/2$.
  Therefore, for $n\ge n_0$ we obtain
\DEQS {\EE^{\bar \PP}} \sup_{0\le s\le t}| \CG _{\ep}^n (s)| &\le
& 2.
\EEQS 
Again applying the It\^o formula and the considerations above we
obtain
\DEQS \lqq{
 {\EE^{\bar \PP}} \lk|\CG_{\ep} ^n ({t_n})-1\rk|  \le  \frac{C(r_1)}{C_\sigma}{\EE^{\bar \PP}}  \int_0 ^ {t_n} \  {\lk| \CG^n _{\ep}(s-)  \rk|}\lk|  v^n _\ep(s) \rk|^2  \, ds }
&&\\
&\le & \frac{C(r_1)}{C_\sigma} \,\,{\EE^{\bar \PP}}{ \sup_{0\le s\le t_n}}| \CG _{\ep}^n
(s)|\,
 \int_0 ^ {t_n}  \lk| v^n _\ep (s) \rk| ^2 \, ds,
\\
&\le &
 \frac{C(r_1)}{C_\sigma} 2 \ep a_n^2.
\EEQS
\noindent Going back to Ansatz \eqref{diffquotient1} and taking
the limit, it follows that there exists some constants $C_1,
C_2>0$ and some $n_0\in\NN$, such that for all $n\ge n_0$
\DEQS \lqq{ \lk| \EE ^ {\bar \PP}\lk[\phi_n\lk( u({t_n}, x+\ep h)
\rk)\rk]
-\EE ^ {\bar \PP}\lk[ \phi_n \lk( u({t_n},x)\rk)\rk]\rk| } && \\
&\le &
 \lk\{
 \frac {C_1} {a_n} \,  \lk(
\int_0^{t_n}|v^n _\ep(s)|^2\,ds\rk) ^\frac 12  + C_2\,
\|\phi\|_\infty\; 2\ep^2 a_n^2  \rk\}. \EEQS
Hence,
 we have
 \DEQS
 &\le & \lk\{
 \frac {C_1 \ep a_n }{a_n} + C_2  \ep a_n\rk\} .
 \EEQS
 Taking the limit $n\to\infty$ we get
 $$\limsup_{n\to\infty} \sup_{y\in B(x,\ep)} \lk\|
\CP_{t_n}(x,\cdot)-\CP_{t_n}(y,\cdot)\rk\|_{d_n}\le C \ep.$$
Taking the limit $\ep\to 0$, the assertion follows.
\end{proof}


\del{\section{Conclusions and additional remarks}

Tracing the proof of ...
}

\section{An Example - the damped wave equation with boundary noise}\label{dampedwave}\label{sec:4}

As mentioned in the introduction, as example we consider an
elastic string, fixed at one end and perturbed at the other end by
a L\'evy noise.
 Mathematically, the system can be formulated as damped wave equation with boundary L\'evy noise.

Throughout this section suppose that we are given a filtered
probability space $(\Omega, \mathcal{F}, \mathbb{F}, \mathbb{P} )$
such that the filtration $\mathbb{F}=(\mathcal{F}_t)_{t\ge 0}$
satisfies the usual condition. On this probability space we assume
that we are given a real valued  L\'evy process $L$. Let  $T>0$ and $\alpha\in\RR$.
We
consider the system
\DEQSZ\label{damped1} \lk\{ \barray u_{tt}(t,\xi) - \Lambda \,
u(t,\xi) + \alpha \, u_\xi(t,\xi) &=& 0,\quad t\in (0,T),\, \xi\in
(0,2 \pi ),
\\
u(t,0)&=& 0 ,\quad t\in (0,T),
\\ u_\xi(t,2 \pi ) &=& \sigma(u(t)) \dot{L}_{t},\quad t\in (0,T),
\\ u(0,\xi)&=& u_0(\xi),\quad u_t(0,\xi) = u_1(\xi),\quad
\earray\rk.
\EEQSZ where $\Lambda=\Delta$ the Laplacian and $\dot{L}$ is the
Radon Nikodym derivative of a real valued  L\'evy process with
characteristic measure $\nu$, $u_0\in H_0^1(0,2 \pi )$ and $u_1\in
L^2(0,2 \pi )$. Here we have set
$\sigma(u(t))=\log(2+|u(t)|_{L^2(0,2\pi)})$ for any $t\in (0,T)$.

Equation \eqref{damped1} can be reformulated as a evolution equation of order one. 
Henceforth, let us introduce  the Hilbert space
$\HC=D(\Lambda^\frac{1}{2})\times L^2(\mathcal{O})$ equipped with the scalar product $$ \langle w,z \rangle_\HC =\langle \Lambda^\frac{1}{2}w_1, \Lambda^\frac{1}{2}z_1\rangle +
\langle w_1, w_2\rangle, \quad  w=\begin{pmatrix} w_1\\ w_2\end{pmatrix}\in\HC \mbox{ and }
z=\begin{pmatrix} z_1\\ z_2\end{pmatrix}\in\HC,
 $$
where $\langle\,\cdot\, ,\,\cdot\, \rangle$ denotes the scalar product on $L^2(\mathcal{O})$.
Define  an operator $\mathcal{A}$ with domain $D(\mathcal{A})=D(\Lambda)\times D(\Lambda^\frac{1}{2})\to \HC$ by
\begin{equation}
 \mathcal{A}\begin{pmatrix}z_1\\
z_2
 \end{pmatrix}
=\begin{pmatrix}
  0&& I\\
-A&& 0
 \end{pmatrix}
\begin{pmatrix}
 z_1\\z_2
\end{pmatrix},
\end{equation}
and $\CB_\alpha :\HC\to \HC $ by
\DEQS \CB_\alpha  \pmat z_1\\z_2\epmat &=& \pmat 0\\
\alpha z_1\epmat. \EEQS
It is not difficult to prove that $\CA$ generates a $C_0$
semigroup $(\CS(t))_{t\ge 0}$ on $\HC$. To be more precise, if $\{
\lambda_n=:n\in\NN\}$ are the eigenvalues and $\{\phi_n:n\in\NN\}$
the eigenfunction of $A$, then $\{ \mu_n:n\in\ZZ\}$ with
$\mu_n=\sqrt{|\lambda_n|}$, $\mu_{-n}=\mu_n$, $n\in\NN$, are the
eigenvalues and
$$
\lk\{ \psi = \frac 1 { \sqrt{2}}\, \pmat {1\over \mu_n } \phi_n\\
\phi_n\epmat : n\in\mZ\rk\},
$$
are the eigenfunction of $\CA$ (see \cite[Proposition]{tucsnak}).
The semigroup $\CS$ can be written as
$$
\CS(t) \pmat f\\g\epmat = {1\over \sqrt{2}} \sum_{n\in\ZZ}
e^{\mu_n} \lk( \mu_n \la {df\over dx},{ d\phi_n\over dx}\ra_{L^2
([0,1])}+\la g, \phi_n\ra _{L^2 ([0,1])}\rk) \psi_n, \quad \pmat
f\\g\epmat\in\CH.
$$

To rewrite \eqref{damped1} as a stochastic evolution equations on the Hilbert space $\HC$ we need to find a way of transforming the nonhomogeneous boundary
conditions in \eqref{damped1} to homogeneous one. Therefore we introduce the  operator $D_{B,\gamma}$.
For every $a\in\RR$, $v=D_{B,\gamma}\, a$ is a solution to the problem
\DEQS
\lk\{ \barray A v(\xi)&=& \lambda v(\xi),\quad \xi \in \CO,
\\
v_\xi(2 \pi ) &=& a, \quad v_\xi(0)=0. 
\earray\rk.
\EEQS
By a short calculation it follows that given $a\in\RR$,
$$v =D_{B,1}(\xi)={a\over e^{2\pi}-e^{-2\pi}}\, (e^{-\xi} +e^{\xi}),\quad  \xi\in [0,2 \pi ].
$$
Following the approach in  \cite{Maslowski} and
\cite{PESZAT+ZABZCYK} we see that \eqref{damped1} can be
transformed to the following
\begin{equation}\label{mod-damped12}
\begin{cases}
dX=\left(\mathcal{A}+\mathcal{B}_\alpha\right)X(t)  dt + (\CA-\CI) 
\pmat 0\\ D_{B,1}\Big(\sigma\circ\Pi_1( X(t))\Big) d {L}\epmat,\\
X(0)=X_0,
\end{cases}
\end{equation}
where $X=(u,\dot{u})^T$, $X_0=(u_0,u_1)^T$. Here $\Pi_1 X=u$
denotes
the projection from $\mathcal{H}$ onto $D(\Lambda^\frac 12)$. 

From now on we will work with
\eqref{mod-damped12}. \del{, whose solution is defined . Now we
define explicitly what we mean by solution of
\eqref{mod-damped12}.
\begin{definition}
 A stochastic process $\{X(t): t\in [0,\infty)\}$ is a mild solution of \eqref{mod-damped12} if
 \begin{itemize}
  \item $X(t)$ is $\mathcal{F}_t$-measurable for any $t\ge 0$,
  \item for any $T\ge 0$, $\int_0^T \EE \lVert X(s)\rVert^2_\mathbf{H} ds<\infty $,
  \item we have
  \begin{equation}
   X(t)=e^{-t\mathcal{A}}X_0 + \int_0^t e^{-(t-s)\mathcal{A} }\mathbf{B} d\mathbf{L},
  \end{equation}
  for any $t\in [0,T], \,\, T\ge 0$  with probability one.
 \end{itemize}

\end{definition}}
First, note that by mimicking the proof of \cite[Theorem
15.7.2]{PESZAT+ZABZCYK} (see also \cite{Maslowski}, \cite{DZ92})
one can show that Problem \eqref{mod-damped12} is well posed.
Moreover, if
 $\int z^2 \nu(dz)<\infty$, then \eqref{mod-damped12} has a
unique mild solution which is a Markov-Feller process. In
particular, the family of operators 
$(\CP_t)_{t\ge 0}$
defined by
\DEQSZ\label{def-mark-semi} \CP_t\phi(x) := \EE \phi(X ),\quad
\phi\in C_b(\HC),\; t\ge 0.
\EEQSZ is indeed a semigroup on $C_b(\HC)$. By means of Theorem
\ref{theo-asf} and Lemma \ref{nondeg2}  the following result can
be obtained.
\begin{theorem}\label{nondeg-uni-heat}
 There exists a time $T>0$ such that for any $C>0$ and $\rho>0$, there exists a $\kappa>0$ such that 
for any $x,y\in \CB_{L^2(\CO)}$ we have $$\CP_T^\ast \delta_x(
\CD_{L^2(\CO)} (\rho,y))\ge \kappa,$$ where $\CD_{\mathcal{H}}
(\rho,y)=\{ z\in \mathcal{H}: |z-y|_\mathcal{H}\le \rho\}$.
\end{theorem}

\begin{coro}\label{ex-uni-inv}
If $\alpha>0$ then the Markovian semigroup $(\CP_t)_{t\ge 0}$
defined by \eqref{def-mark-semi}
 has at most one invariant measure.
\end{coro}

\begin{remark}
Since the system is approximate controllable with vanishing energy in case $\alpha=0$ it is asymptotically strong Feller
also for $\alpha=0$.
\end{remark}

We need to show some facts which are essential for the results in
the previous sections to be applicable in for our example. First
note, that the following system
 \DEQSZ\label{damped1-ex} \lk\{ \barray
u_{tt}(t,\xi) - \Lambda \, u(t,\xi) + \alpha \, u(t,\xi) &=&
0,\quad t\in (0,T),\, \xi\in (0,2 \pi ),
\\
u(t,0)&=& 0 ,\quad t\in (0,T),
\\ u_\xi(t,2 \pi ) &=& v(t),\quad t\in (0,T),
\\ u(0,\xi)&=& u_0(\xi),\quad u_t(0,\xi) = u_1(\xi),\quad
\earray\rk.
\EEQSZ
with control $v\in L^ 2 ([0,\infty);\RR)$ is approximate null
controllable with vanishing energy. This statement is proved in
the following Lemma. \del{ (see e.g.\ \cite[Example
11.2.6]{tucsnak}. Moreover, it is shown in \cite[Proposition
6.3]{priola-3}, that the system \eqref{damped1-ex} is null
controllable (and even exact controllable) with vanishing energy.}

\begin{lemma}\label{wave-control}
 The system \begin{equation}\label{damp-control}
 \begin{cases}
             \frac{\partial X(t) }{\partial t}=(\mathcal{A}+\CB_\alpha)X(t)+(\CA-\CI) 
\pmat 0\\ D_{B,1} v(t)\epmat ,\\
             X(0)=X_0,
             \end{cases}
            \end{equation}
is approximate null controllable with vanishing energy.
\end{lemma}
\begin{proof}
It was proved in {\cite[Section 2.4]{Coron}}  (see also e.g.\
\cite[Example 11.2.6]{tucsnak}) that \eqref{damped1} is exactly
controllable at any time $T$, hence it is null controllable.
Thanks to \cite[Theorem 2.45]{Coron} it is approximately null
controllable. Now it remains to prove that it is approximate null
controllable with vanishing energy. For this purpose we mainly
follow the idea in \cite{priola-3}. \del{Since $A$ is positive and
self-adjoint then there exists an orthonormal basis of of
eigenfunctions $\{e_k: k=1,2,3, \dots\} $ which corresponds to the
increasing eigenvalues $\{\lambda_k: k=1,2,3,\dots\}$. By
definition of $\mathcal{A}$ we can check that $E_k=\begin{pmatrix}
                                                       E_k^{(1)}\\E_k^{(2)}
                                                      \end{pmatrix}$ is an eigenfunction corresponding to the eigenvalue $\mu_k$ if it satisfies
 \begin{equation}
  \begin{cases}
   E_k^{(2)}=\mu_k E_k^{(1)},\\
   -A E_k^{(1)}=\mu_k E_k^{(2)}.
  \end{cases}
 \end{equation}}
\del{We easily deduce  that the  family of eigenfunctions of
$\mathcal{A}$ $$\biggl\{E_k=\frac{1}{\sqrt{2} }\begin{pmatrix}
                                                 \frac{e_k}{\sqrt{\lambda_k} }\\
                                                 \pm i \sqrt{\lambda_k} e_k
                                                \end{pmatrix}
: k=1,2,3,\dots \biggr\}$$  corresponding to the eigenvalues
$\{\mu_k=\pm i \sqrt{\lambda_k}: k=1, 2, 3, \dots\}$ forms an
orthonormal basis of $\HC $.} Let us write  $\HC$ as the direct
sum $\HC =\HC _s\oplus \HC _u$ where $\HC _u=\{ 0\}$ and $\HC
_u=\HC $. Therefore we see that \cite[Hypothesis 1.1]{priola-3}
are satisfied in our case. Moreover, since $\mathcal{A}$ is the
infinitesimal generator of a strongly continuous semigroup of
contractions we can deduce from \cite[Chapter 1, Corollary
3.6]{Pazy:83} that the spectrum $\sigma(\mathcal{A})$ is contained
in $\{\lambda\in \mathbb{C}: \Re(\lambda)\le 0 \}$. This fact
implies that $S(A)=\sup \{\Re(\lambda): \lambda \in
\sigma(\mathcal{A}) \}\le 0$. Therefore we can deduce from
\cite[Theorem 1.1]{priola-3} that \eqref{damp-control} is null
controllable with vanishing energy.
\end{proof}
Now we are ready to prove the existence and uniqueness of the
invariant measure
\begin{proof}[Proof of Theorem \ref{ex-uni-inv}]
To show the existence of  the invariant measure  we can argue
exactly as in Theorem \eqref{exinv}.

It remains to show the  uniqueness of the invariant measure. Owing
to the Lemma \ref{wave-control} and Theorem \ref{theo-asf} the
semigroup $\CP_t$ is asymptotically strong Feller. By
\cite[Corollary 3.17]{mat-hai-2006} we know that if the semigroup
is asymptotically strong Feller and there exists a point $x\in
\HC$ such that $x\in supp(\rho)$, whenever $\rho$ is an invariant
measure of $(\CP_t)_{t\ge 0}$, then the Markovian $(\CP_t)_{t\ge
0}$ semigroup admits almost one invariant measure. Therefore, we
have to show that there exists a point $x\in H$ such that for any invariant measure $\rho$,  $x\in supp(\rho)$, i.e.\ for all $\kappa >0$,
$\nu( \CR_\HC(\kappa))>0$.
Since null {controllability} implies {approximate null}
controllability, Theorem \ref{nondeg2} can be applied and there exists a time $T>0$
such that for all $C>0$ and $\gamma>0$ there exists a $\kappa>0$ with
\DEQSZ\label{inv-2-dwe}
\PP\lk( u(T,x)\in\CR_\HC(\gamma)\rk), \quad x\in \bar \CR_\HC(C). \EEQSZ
It remains to show \eqref{inv-1}. In particular, we should check
that there exists a constant $C>0$ such that \DEQSZ\label{inv-11}
\inf_{\{\mu \mbox{ \small  is an invariant measure}\}} \mu(\bar
\CR_\HC( C)) > 0. \EEQSZ It follows that $0\in supp(\mu)$ by the
following observations. First, since $\mu$ is invariant we have
$$
\mu\lk( \CR_\HC(\gamma)\rk) \ge  \mu\lk( \bar \CR_\HC(C) \rk)\cdot
\inf_{\{\mu \mbox{ \small is invariant measure}\}} \mu(\bar
\CR_\HC( C)).
$$
Now, the estimates \eqref{inv-11} and \eqref{inv-2} give the
assertion from which we easily complete the proof of the Theorem
\ref{ex-uni-inv}.
\end{proof}

\del{ Now we are ready to prove the existence and uniqueness of
the invariant measure
\begin{proof}[Proof of Theorem \ref{ex-uni-inv}]
To show the existence of  the invariant measure  we can argue exactly as in Theorem \eqref{exinv}.

Let us proceed to uniqueness of the invariant measure. Since
\eqref{damped1-ex}  is approximate controllable with vanishing
energy, one knows by Theorem \ref{theo-asf} that the semigroup
$\CP_t$ is asymptotically strong Feller. By \cite[Corollary
3.17]{mat-hai-2006} we know that if the semigroup is
asymptotically strong Feller and there exists a point $x\in \HC$
such that there exist a point $x\in\HC$ with $x\in supp(\mu)$,
whenever $\nu$ is an invariant measure of $(\CP_t)_{t\ge 0}$, then
the Markovian $(\CP_t)_{t\ge 0}$ semigroup admits almost one
invariant measure. Therefore, we have to show that  $x\in
supp(\mu)$, i.e.\ for all $\rho >0$, $\nu( \CR_\HC(\rho))>0$.
Since null {controllability} implies  {approximate null}
controllability, Lemma \ref{nondeg2} can be applied and for all
$\rho>0$ and for all $x\in \bar \CR_\HC(C)$ there exists a time
$T_\rho>0$ such that \DEQSZ\label{inv-2-dwe}
u(T_\rho,x,\eta)\in\CR_\HC(\rho). \EEQSZ
It remains to show \eqref{inv-1}.
In particular, we should check that there exists a constant $C>0$ such that
\DEQSZ\label{inv-11}
\inf_{\{\mu \mbox{ \small  is an invariant measure}\}} \mu(\bar \CB( C)) > 0.
\EEQSZ
It follows that $0\in supp(\mu)$ by the following observations.
First, since $\mu$ is invariant we have
$$
\mu\lk( \CR_\HC(\gamma)\rk) \ge  \mu\lk( \bar \CR_\HC(C) \rk)\cdot
\inf_{\{\mu \mbox{ \small is invariant measure}\}} \mu(\bar \CB(
C)).
$$
Now, the estimates \eqref{inv-11} and \eqref{inv-2} give the assertion from which we easily complete the proof of the Theorem \ref{ex-uni-inv}.
\end{proof}
}

\appendix

\section{Technical Preliminaries}\label{sec:6}

In this section we will show that one can find a transformation
$\theta:\RR\times \RR\to\RR$ such that for any $v\in\RR$,
$$
\int_\RR z (\nu_\theta-\nu)(dz)=v,
$$
and $\nu_\theta$ is a L\'evy measure. Here
$$\nu_\theta:\CB(\RR)\ni B \mapsto \int_\RR 1_B(\theta(v,z))\,
\nu(dz).
$$

 In order to find the transformation,
 it is convenient to switch the representation of
the Poisson random measure as in the beginning of the proof of
Theorem \ref{theo-asf}. Let $\nu$ be a L\'evy measure satisfying
Hypotheses \ref{nondeg1}. Let
\DEQS 
c:
 \RR ^+ \ni r \mapsto   \sup_{\rho>0} \lk\{\int_\rho^\infty k(s)\, ds \ge r \rk\} \,.
\EEQS
%
\noindent
To analyse the effect of the perturbation, we
define a function $\wt  $ by \DEQSZ
\\
\nonumber\label{defk}
 [0,\infty)
 \ni K \mapsto \wt   (K) :=    
\int_{\RR^+} \lk( c(z) - c(z+  \hh (K,z))\rk) \; dz\in\CH, \EEQSZ
where $\hh :\RR ^ +\times \RR^ + \to\RR ^ +$ is defined by
{\DEQSZ\label{v0s} \hh (K,z) :=
\bcase
K z ^ {-\beta_1} & z \in \lk(  K^\neu r_1,K^ \neu 2 r_1\rk) \mbox{ and } K\ge 1,
\\
0 & z \not \in \lk(  K^\neu r_1-\frac 14,K^ \neu 2 r_1+\frac 14\rk) \mbox{ and } K\ge 1,
\\
C 
z ^ {-\beta_2} & z \in \lk(
 r_1,  r_1(1+ K^{\gamma_2} )\rk) \mbox{ and } K< 1,
\\
0& z \not\in \lk(
 r_1-\frac 14,  r_1(1+ K^{\gamma_2} )+\frac 14\rk) \mbox{ and } K< 1,
\\
\lqq{ \mbox{differentiable interpolated elsewhere}},
\ecase
\EEQSZ
and
$$\beta_1 =\frac{3-2 \alpha }{\alpha  (3 \alpha -5)}
   ,\quad \beta_2=-1, 
, 
$$
and $$\neu=5 \alpha -3 \alpha ^2, \quad \gamma_2 =- {\alpha\over 2}.
$$
The constant $C>0$ has to be chosen in such a way, that $K\mapsto
\wt  (K)$ is continuously in $K$. Moreover, let \DEQSZ
\label{defw} \wt  :\CH^+_0\ni x\mapsto  \wt   (x). \EEQSZ
\begin{lemma} \label{trans_ex}
The function $\wt  :\CH_0^+\to\CH_0^+$ is invertible. \del{In
particular, there exists a constant $C=C(r_1)>0$ such that
$$
\wt  ^{-1}(x) \le C\, \lk(1+|x|^{ (\beta+1)\alpha\over \alpha-1}
\rk).
$$}
\end{lemma}

\begin{proof}
We will show that  there exists a
function $\kappa:\RR^+\to\RR^+$ 
such that $\kappa( \wt  (x) )= \wt  (\kappa(x))=x$ for all $x\in
\RR^+_0$ with $ \wt  (x)
> 0$.
\del{Moreover,
there exists a constant $C>0$ such that 
\DEQSZ\label{ki-in} K(x) &\le& C\, \lk( 1+ |x|^{
(\beta+1)\alpha\over \alpha-1}\rk). \EEQSZ}

\medskip
Hence, we have to show that for $K\in\RR^+_0$ $\wt  (K)$ defined
by \DEQSZ\label{fidef}
 \RR^+ \in K\mapsto \wt  (K) = \int_0 ^ \infty  \lk[ c(z) - c(z+ \hh (K,z))\rk]   \; dz
\EEQSZ is invertible on $\RR^+_0$. We start by verifying the
following properties
\begin{enumerate}
\item $\wt  (K)\in\RR^+_0$; \item  the function $\RR^+_0\ni
K\mapsto \wt(K) \in\RR^+ _0$ is continuous.
\item the function $\RR^+\ni K\mapsto \wt(K) \in\RR^+_0 $ is injective. 
\item the function $\RR^+\ni K\mapsto  \wt(K) \in\RR^+ _0$ is surjective. 
\end{enumerate}
It follows, in particular,  from (2), (3) and (4), that the
function $\wt$   is invertible.
%
%


In fact, (1) is clear by the definition of $c$.
 In
order to show Item (2) we take into account that  the function
 $\RR ^+_0 \ni K\mapsto  \wt  (K) \in\RR ^ +_0$ is strictly decreasing and
  continuous. In order to show Item
(3) we will show, that
 $\lim_{K\to\infty}  \wt  (K) =\infty$. Since $\wt  (0)=0$ and $\wt  $ is continuous on $\RR^+_0$, the claim follows.

Firstly, we consider the case $K>1$.
Here we have
\DEQS \lqq{ \wt  (K)=\int_0 ^ \infty \lk[ c(z) - c(z+\hh
(K,z))\rk] \; dz =
 \int_{ r_1K^\neu} ^ {K^\neu 2 r_1} \int_z ^ { z+\hh (K,z)} \;  {d\over dy}\,c(y) \; dy\; dz} &&
\\
&=& \int_{ K^ \neu r_1} ^  {K^ \neu 2 r_1}
 \int_z ^ {z+K\,
 \hh (K,z)}\; {d\over dy}\, c(y) \; dy\; dz \ge \int_{K^ \neu r_1} ^  {K^ \neu 2 r_1}
 \hh (K,z)\; {d\over dy}\, c(z+ \hh (K,z)) \; dz . \EEQS
Hypotheses  \ref{nondeg1}  give for $\tilde \neu =
\gamma_1(1+\beta_1)-1= \frac{17 \left(\frac{\alpha -12}{17 \alpha
}+1\right)
   \alpha }{3 \alpha -2}-1=5$ 
\DEQS\ldots &\ge&\delta_0 K\; \int_{ r_1K^ \neu } ^ {K^ \neu 2
r_1} { z ^{-{\beta_1}}\over \lk( z+{K\over z ^ {\beta_1}}\rk) ^
{\frac 1\alpha+1} }\; dz \ge \delta_0
 K\; \int_{ r_1K^ \neu } ^ {K^ \neu 2 r_1}  { z ^{-{\beta_1}+{\beta_1} (
{\frac 1\alpha +1}) }\over \lk( z ^ {1+{\beta_1}}+K\rk) ^ {\frac
1\alpha +1} }\; dz
\\
&=& \delta_0\, K \, K ^ {- \frac {1+\alpha}\alpha }
 \int_{ r_1K^ \neu } ^  {K^\neu  2 r_1}  { z ^{{\beta_1} \over \alpha}
 \over \lk( {z ^ {1+{\beta_1}} \over K} +1  \rk) ^
{\frac 1\alpha +1} }\; dz
\\ &= &  K ^{-\frac 1\alpha}
 \int_{r_1^ {1+{\beta_1}} K^{\tilde \neu}  }  ^  {  (2r_1)^{1+{\beta_1}}  K^{\tilde \neu}  }
   { \lk(K u  \rk) ^{{1\over 1+{\beta_1}}\, {\beta_1} \over \alpha}
 \over \lk( {u } +1  \rk) ^
{\frac 1\alpha +1} }\; K ^{\frac 1{{\beta_1}+1}} u  ^{\frac 1{{\beta_1}+1} -1} du
\\ & =& K ^{ - (1+{\beta_1}) +{\beta_1} +\alpha   \over \alpha (1+{\beta_1})}
\int_{{r_1^ {1+{\beta_1}}K^{\tilde \neu }}}
^{(2r_1)^{1+{\beta_1}}K^{\tilde \neu } } { u ^{ {\beta_1} +\alpha
- ({\beta_1}+1)\alpha  \over ({\beta_1}+1)\alpha}
 \over \lk( {u } +1  \rk) ^
{\frac 1\alpha +1} }\;  du \\ &%
 = &
 K ^{ \alpha -  1  \over \alpha (1+{\beta_1})}
\int_{{r_1^ {1+{\beta_1}}K^{\tilde \neu }}} ^
{(2r_1)^{1+{\beta_1}}K^{\tilde \neu } } {  u  ^{ {\beta_1}
(1-\alpha ) \over ({\beta_1}+1)\alpha}
 \over \lk( {u } +1  \rk) ^
{\frac 1\alpha +1} }\;  du.
\EEQS
First, note that we have 
$\tilde \neu=-3 \alpha ^2+7 \alpha -4<0$ for $1< \alpha\le 2$ and $(1+\beta_1)=\frac{6 (3 \alpha -2)}{17 \alpha }>0$.
That means, we have for
all $u\ge r_1^ {1+{\beta_1}}K^{\tilde \neu}$ 
%
$$
{1\over (1+u) ^{\frac 1 \alpha +1} }\ge  { ( r_1^{(1+{\beta_1})}K^
{\tilde \neu} ) ^{{1+\alpha \over \alpha}} \over
(1+r_1^{(1+{\beta_1})}K^ {\tilde \neu} ) ^{\frac 1 \alpha +1} }\,
u^{-(\frac 1 \alpha +1)}\ge
   u^{-(\frac 1 \alpha +1)}. 
$$
Integration and substituting of ${\beta_1}$ give for $K>1$
\DEQSZ\nonumber\label{integration-gives} \lqq{
\wt  =\int_0 ^ \infty \lk[ c(z) - c(z+K \hh (z))\rk] \; dz} &&
\\  \nonumber&\ge &  K ^{ \alpha -  1  \over \alpha (1+{\beta_1})}   { r_1 ^{(1+{\beta_1}){1+\alpha \over \alpha}}\over (1+r_1 ^{1+{\beta_1}}) ^{\frac 1 \alpha +1} }
\int_{{r_1^ {1+{\beta_1}}K^{\tilde \neu}}  } ^
{(2r_1)^{1+{\beta_1}}K^{\tilde \neu} }  {  u  ^{ { {\beta_1}
(1-\alpha ) \over ({\beta_1}+1)\alpha}- \frac 1 \alpha -1} }\; du
\\  \nonumber
&=& K ^{ { \alpha -  1  \over \alpha (1+{\beta_1})}  } { r_1
^{(1+{\beta_1}){1+\alpha \over \alpha}}\over (1+r_1
^{1+{\beta_1}}) ^{\frac 1 \alpha +1} } \int_{{r_1^
{1+{\beta_1}}K^{\tilde \neu}} } ^ {(2r_1)^{1+{\beta_1}}K^{\tilde
\neu} } {  u ^{-{\alpha {\beta_1} + 1 \over ({\beta_1}+1)\alpha}-1
} }du
\\  \nonumber &= & C\,
 K ^{ { \alpha -  1  \over \alpha (1+{\beta_1})}  + \tilde \neu {1+\alpha\over \alpha}}
  { r_1 ^{(1+{\beta_1}){1+\alpha \over
\alpha}}\over (1+r_1 ^{1+{\beta_1}}) ^{\frac 1 \alpha +1} }
r_1^{ -(1+{\beta_1}) {{\alpha {\beta_1} + 1  \over
({\beta_1}+1)\alpha}}}\lk( C (2K)^{ -{\tilde \neu
(\alpha{\beta_1}+1)\over\alpha(1+{\beta_1})}}\rk)
\\  \nonumber &\ge &
C\,  K ^{{ \alpha-1-\tilde \neu (\alpha{\beta_1}+1)\over
\alpha(1+{\beta_1})}}
 { {r_1^ {{\beta_1}+\alpha\over
\alpha}}\over (1+{r_1^ {1+{\beta_1}}}) ^{\frac 1 \alpha +1} }
\\  \nonumber &\ge &
C\,  K ^{\alpha -1}
 { {r_1^ {{\beta_1}+\alpha\over
\alpha}}\over (1+{r_1^ {1+{\beta_1}}}) ^{\frac 1 \alpha +1} }
 .
\EEQSZ
%
%
It follows that \DEQSZ K\le r_1^{\alpha\beta_1+1\over \alpha} \wt
^{\frac{1 }{\alpha -1 }}, \EEQSZ and, therefore
$$
\lim_{K\to\infty} \int_{r_1} ^ \infty \lk[ c(z) - c(z+\hh
(K,z))\rk] \; dz =\infty.
$$
 From (1), (2) and (3) it follows that the function $\wt :\RR^+\to\RR^+$ defined by
\eqref{fidef}
  is invertible.
For any $z\in \RR^+$ let us write
$\kappa( z)=K$ iff $\wt(K)=z$. \label{def-kappa}

It remains to investigate the rate of grow  for $0<K\le 1$. Here, we have

\DEQS \lqq{ \int_0 ^ \infty \lk[ c(z) - c(z+ \hh (K,z))\rk] \; dz
}&&
\\ &%
 = &
\int_{r_1}^{r_1(1+K)^{\gamma_2}} \rho(K,z) \frac{d}{dz}c(z+\rho(K,z)) dz \;  du, \EEQS
which implies the existence of a positive constant $C(r_1, \alpha, \gamma_2,\beta_2)$ such that
\DEQS
\lqq{  \int_0 ^ \infty \lk[ c(z) - c(z+ \hh (K,z))\rk] \; dz
}
\\
& =& C(r_1, \alpha, \gamma_2,\beta_2)
\int_{r_1}^{r_1(1+K^{\gamma_2})}\frac{z^{-\beta_2}}{(z+z^{-\beta_2})^{1+\frac{1}{\alpha}}}dz
\\
& =& C(r_1, \alpha, \gamma_2,\beta_2)
\int_{r_1}^{r_1(1+K^{\gamma_2})}\frac{z^{-\beta_2{\alpha-1\over \alpha}
}}{(z^{1+\beta_2}+1)^{1+\frac{1}{\alpha}}}dz.
\EEQS
By changing of variables 
 we get that
\DEQS
\lqq{\int_0 ^ \infty \lk[ c(z) - c(z+ \hh (K,z))\rk] \; dz}
\\
& \ge&  C(r_1, \alpha, \gamma_2,\beta_2)
\int_{{r_1}^{1+\beta_2}}^{r_1^{1+\beta_2}(1+K^{\gamma_2})^{(1+\beta_2)}}\frac{u^{-\frac{\beta_2(\alpha-1)}{\alpha(\beta_2+1)}} }{(u+1)^{1+\frac{1}{\alpha}}} u^{\beta_2\over 1+\beta_2}  du
\\
& \ge&  C(r_1, \alpha, \gamma_2,\beta_2)
\int_{{r_1}^{1+\beta_2}}^{r_1^{1+\beta_2}(1+K^{\gamma_2})^{(1+\beta_2)}}\frac{u^{\frac{\beta_2}{\alpha(\beta_2+1)}}
 }{(u+1)^{1+\frac{1}{\alpha}}} du.
\EEQS
Since  ${r_1}^{1+\beta_2}\le  u\le r_1^{1+\beta_2}(1+K^{\gamma_2})^{(1+\beta_2)}$, we get
\begin{equation*}
\begin{split}
 &  \int_0 ^ \infty \lk[ c(z) - c(z+ \hh (K,z))\rk] \; dz \\
  & \quad \ge C(r_1, \alpha, \gamma_2,\beta_2)
 \frac{1}{(1+r_1^{1+\beta_2}(1+K^{\gamma_2})^{(1+\beta_2)})^{1+\frac{1}{\alpha} } } \int_{{r_1}^{1+\beta_2}}^{[r_1(1+K^{\gamma_2})]^{(1+\beta_2)} }
 u^{-{1+\alpha +\alpha \beta_2\over \alpha (\beta_2+1)}}
 du,
\end{split}
\end{equation*}
which implies that
\begin{align*}
\lqq{ \int_0 ^ \infty \lk[ c(z) - c(z+ \hh (K,z))\rk] \; dz } &
\\
&\ge  C(r_1, \alpha, \gamma_2,\beta_2)
 \frac { r_1^ {-{1\over \alpha}}  }  {(1+[r_1(1+K^{\gamma_2}) ]^{\beta_2+1} )^{1+\frac{1}{\alpha}} } K^{-{\gamma_2}\, {1+\alpha +\alpha \beta_2\over \alpha }}\\
& \ge  C(r_1, \alpha, \gamma_2,\beta_2)  K^{-{\gamma_2}\, {1+\alpha +\alpha \beta_2\over \alpha }}\ge  C(r_1, \alpha, \gamma_2,\beta_2)  K^{\frac 12 } .
\end{align*}
This proves Lemma \ref{trans_ex}. 
%
\end{proof}
The following two corollaries are following.
\begin{coro}\label{theta-exx}
%
Under Hypothesis \ref{nondeg1} for any $\tilde r>r_1$
and $v\in\RR^ +_0$ there exists a number $K>0$ 
such that
\DEQSZ \label{id-001}      
 \int_{\tilde r} ^\infty  
    \lk[ c(r)-c(\theta(K,r)) \rk]
    \, dr\, = v  . 
\EEQSZ
 Moreover, there exists a constant $ C(\tilde r)>0$
\del{given by
$$
C(r_1) = \lk( { r_1 ^{\frac 1 \alpha +1}\over (1+r_1) ^{\frac 1
\alpha +1} }  r_1 ^{- {2\alpha +1\over 1+\alpha}}\rk)
^{\alpha(1+{\beta_2})\over \alpha -1}
$$}
 such that $v\in\RR_0^+$ 
$$
\int_{\tilde r}^\infty |\hh_z(\kappa(v),r)|\, dr \le C(\tilde r)\,
v ^{2}.
$$
\end{coro}
\del{If $c$ is strictly decreasing, the number $K$ is unique. That
means, we can define a function $\CK:\RR^+_0\to \RR^+_0$ with
$\CK(v)=K$, if identity \eqref{id-001} holds.} Taking into account
the negative jumps, we will define the following transformation.
\begin{coro}\label{theta-ex}
Under Hypothesis \ref{nondeg1} for any $\tilde r>r_1$ and
$v\in\RR$ there exists a transformation $\theta: \CH\to \CH$ given
by
$$\theta(v,z):= \bcase z+\hh(|v|,z)& \mbox{ if } v\ge 0,
\\-z-\hh(|v|,z)& \mbox{ if } v< 0,
\ecase
$$
such that
$$     
 \int_{\tilde r} ^\infty  
    \lk[ c(z)-c(\theta(v,z)) \rk]
    \, dz\, = v  . 
$$
Moreover, there exists a constant $ C(\tilde r)>0$ \del{given by
$$
C(r_1) = \lk( { r_1 ^{\frac 1 \alpha +1}\over (1+r_1) ^{\frac 1
\alpha +1} }  r_1 ^{- {2\alpha +1\over 1+\alpha}}\rk)
^{\alpha(1+{\beta_2})\over \alpha -1}
$$}
 such that 
$$
\int_{\tilde r}^\infty |\hh_z(\kappa(|v|),z)|\, dz \le C(\tilde
r)\,| v |^{2}.
$$
\end{coro}

\section{Change of measure formula}\label{sec:7}

Let $\mu$ be a Poisson random measure over $\mathfrak{A}=(\Omega,\PP,(\CF_t)_{t\ge 0},\CF)$
 with compensator $\gamma=\lambda\cdot\lambda$. Let $c:\CH\to\CH$ the transformation
 defined by \eqref{trans-cc1}.
 Let $g\in L^ 2([0,\infty);\RR)$ be a predictable process and
 \del{$$
 \psi (z,s) := z+g(s)  \hh (z),
 $$
and for $v \in L^2([0,\infty); \CH) $ } let \DEQSZ\label{mmmm}
\psi :[0,\infty) \times \CH \ni (s,z) \mapsto z+g(s)  \hh (z)
 \, \in\CH . \EEQSZ
Combining Corollary \ref{theta-ex} and Example  1.9 of \cite{abso}
one can verify the following Lemma.
\begin{lemma}\label{density}
There exists a probability measure $\QQ^ \psi$ on $\mathfrak{A}$
such that the Poisson random measure $\mu_\psi$ defined by
$$
\CB(\CH)\times \CB([0,\infty))\ni A\times I \mapsto \int_I\int_\CH  1_A(\psi(s,z ))\mu(dz,ds)
$$
has compensator $\gamma$. For $t\ge 0$ let $\QQ^\psi_t$,
respectively,  $\PP_t$,  be the projection of $\QQ^\psi$ onto
$\CF_t$, respectively,  of $\PP$ onto $\CF_t$. Then the density
process given by
$$
[0,\infty) \ni t\mapsto \CG(t) := {d\QQ^\psi _t\over  d\PP_t},\quad t> 0,
$$
satisfy
$$
\lk\{ \barray d\CG(t) &=& \CG(t-) \int_\CH (1- \psi_z(z ))
\,(\mu-\gamma)(dz,dt),\phantom{\Big|}
\\&=&
\CG(t-) \int_\CH g(s) \hh _z(z )
\,(\mu-\gamma)(dz,dt),\phantom{\Big|}
\\
\CG(0) &=& 1, \phantom{\Big|}\earray \rk.
$$
where $\hh :\RR ^ +\to\RR ^ +$ is defined by \eqref{v0s} and
$\hh_z$ denotes the derivative of $\hh$.
\end{lemma}
\begin{proof}
The proof is done via the Laplace transform.  Let
$\xi=\{\xi(t):0\le t<\infty\}$ be given by
\DEQS \lk\{ \barray d\xi (t)&=& \int_\CH  c(z) (\mu_\psi -\gamma)
(dz,dt),
\\
\xi(0)&=& 0. \earray \rk. \EEQS
Then under ${\QQ^ \psi}$ the Laplace transform is given by
$$
\EE ^{\QQ^ \psi} e ^{-\lambda  \xi(t)\uu } = e ^{ \int_\CH\lk[ e
^{-\lambda     c(z)\uu  }-1+\lambda     c(z)\uu  \,
\rk]\lambda(dz)}
$$
Rewriting  $\xi $ gives 
\DEQS \lk\{ \barray d\xi(t)  &=&  \del{\int_\CH   \lk[
c(\psi(s,z)) -c(z)\rk](\mu-\gamma) (dz,dt)+ \int_\CH
c(\psi(t,z)) (\mu-\gamma) (dz,dt)\phantom{\Big|}
\\ &&{} +\int_\CH    \lk[ c(\psi(s,z)) -c(z)\rk] \gamma(dz, dt)\phantom{\Big|}
\\ &=& } \int_\CH     c(\psi(t,z)) (\mu-\gamma) (dz,dt) +\int_\CH    \lk[ c(\psi(s,z)) -c(z)\rk] \gamma(dz, dt),\phantom{\Big|}
\\
\xi(0)&=& x_0.\phantom{\Big|} \earray\rk.\EEQS
Let $M _\lambda =\{ M_\lambda(t):0\le t<\infty\}$ be given by
$M_\lambda(t) = e ^{-\lambda  \xi (t)}$, $0\le t<\infty$.
 Now, we will show
that $\EE ^\PP M _\lambda (t)\CG(t) = \EE ^{\QQ^ \psi} e
^{-\lambda  \xi (t)\uu  }$. First $M _\lambda (t)$ solves
%
\DEQS dM _\lambda (t) &=& -\lambda\, \int_\CH  M _\lambda (t-)
\lk[c(\psi(t,z)) -c(z)\rk] \uu \gamma(dz, dt)
\\ && {}+  \int_\CH  M _\lambda(t-)\lk[ e^{-\lambda     c(\psi(t,z))\uu }-1\rk] (\mu-\gamma) (dz,dt)
\\ && {} + \int_\CH   M _\lambda (t-) \lk[ e^{-\lambda     c(\psi(t,z))\uu }- 1+\lambda       c(\psi(t,z))\uu  \rk] \gamma(dz,
dt),
\\ M_\lambda(0)&=& 1.
\EEQS
Therefore,  $\mathcal{Z}_\lambda (t) = M _\lambda (t) \, \CG(t)$
is given
by 
  \DEQS \EE^{\PP}
\mathcal{Z} _\lambda (t) &=&  -\lambda \EE^{\PP} \int_0^t \int_\CH
\mathcal{Z} _\lambda (s-)  \lk[ c(\psi(s,z)) -    c(z)\rk] \uu
\lambda(dz)\, ds
\\
&& {}+  \EE^{\PP} \int_0^t  \int_\CH \mathcal{Z} _\lambda (s-)
\lk[ e^{-\lambda      c(\psi(s,z))\uu }- 1+\lambda
    c(\psi(s,z))\uu  \rk] \lambda(dz)\, ds
\\
&& {}+  \EE^{\PP} \int_0^t \mathcal{Z} _\lambda (s-) \int_\CH \lk[
e^{-\lambda  c(\psi(s,z))\uu }- 1\rk] \lk[ \psi_z(s,z) -1\rk]
\lambda(dz)\, ds
\\ &=&  \EE^{\PP} \int_0^t  \int_\CH \mathcal{Z} _\lambda (s-)\Big[ \lambda
    c(z) \uu  - \lambda       c(\psi(s,z)) \uu +
  e^{-\lambda       c(\psi(s,z))\uu }- 1
\\
&& {} +\lambda      c(\psi(s,z))\uu  +  e^{-\lambda
    c(\psi(s,z))\uu }\psi_z(s,z) - \psi_z(s,z) -e^{-\lambda
    c(\psi(s,z))\uu }+1\Big] \gamma(dz,ds)
\\ &=&
 \EE^{\PP} \int_0^t  \int_\CH \mathcal{Z} _\lambda (s-)\Big[
 e^{-\lambda       c(\psi(s,z))\uu }\psi_z(s,z) - \psi_z(s,z)  +\lambda      c(z) \uu    \Big] \gamma(dz,ds).
\\ &=&
 \EE^{\PP} \int_0^t  \int_\CH \mathcal{Z} _\lambda (s-)\Big[
 e^{-\lambda       c(\psi(s,z))\uu }-1\Big]\psi_z(s,z)  \lambda(dz) ds
 \\ &&{}+   \lambda    \EE^{\PP} \int_0^t  \int_\CH \mathcal{Z} _\lambda (s-)  c(z)  \gamma(dz,ds).
\EEQS Substitution gives
\DEQS \EE^{\PP}  \mathcal{Z} _\lambda (t) &=&  \EE^{\PP} \int_0^t
\int_\CH \mathcal{Z} _\lambda (s-)\Big[
 e^{-\lambda      c(z)\uu } - 1  +\lambda     c(z) \uu   \Big] \gamma(dz,ds).
\EEQS Since
 \DEQS &&\EE ^{\QQ_\psi} \lk[ e^{-\lambda
{\xi(t) \uu }}\rk] = \EE ^\PP \lk[ \CG(t)\, e^{-\lambda {\xi(t)\uu
}}\rk]= \EE ^\PP \lk[ \mathcal{Z} _\lambda (t) \rk]
\\
&=& \exp\lk(  \int_0^t \int_\CH \lk[
 e^{-\lambda      c(z)\uu } - 1  +\lambda      c(z) \uu    \rk]
 \gamma(dz,dt)\rk),
 \EEQS
 the Proposition follows.
\end{proof}

\def\cprime{$'$} \def\cprime{$'$}
  \def\polhk#1{\setbox0=\hbox{#1}{\ooalign{\hidewidth
  \lower1.5ex\hbox{`}\hidewidth\crcr\unhbox0}}}
  \def\polhk#1{\setbox0=\hbox{#1}{\ooalign{\hidewidth
  \lower1.5ex\hbox{`}\hidewidth\crcr\unhbox0}}} \def\cprime{$'$}
  \def\polhk#1{\setbox0=\hbox{#1}{\ooalign{\hidewidth
  \lower1.5ex\hbox{`}\hidewidth\crcr\unhbox0}}}
  \def\polhk#1{\setbox0=\hbox{#1}{\ooalign{\hidewidth
  \lower1.5ex\hbox{`}\hidewidth\crcr\unhbox0}}}
%
%
%


%
%

\end{document}